\documentclass{amsart}
\usepackage{amsmath}
\usepackage{amssymb}
\usepackage{amsfonts}
\usepackage{verbatim}
\usepackage{color}

\numberwithin{equation}{section}
\newtheorem{thm}[equation]{Theorem}
\newtheorem{cor}[equation]{Corollary}
\newtheorem{lem}[equation]{Lemma}
\newtheorem{prop}[equation]{Proposition}
\newtheorem{qu}[equation]{Question}

\newtheorem{defn}[equation]{Definition}

\newtheorem{rem}[equation]{Remark}
\newtheorem{ex}[equation]{Example}

\newcommand{\corref}[1]{Corollary~\ref{#1}}
\newcommand{\defref}[1]{Definition~\ref{#1}}

\newcommand{\exref}[1]{Example~\ref{#1}}
\newcommand{\lemref}[1]{Lemma~\ref{#1}}
\newcommand{\propref}[1]{Proposition~\ref{#1}}
\newcommand{\thmref}[1]{Theorem~\ref{#1}}

\newcommand{\ol}{\overline}

\newcommand{\ra}{\rightarrow}

\newcommand{\invs}{^{-1}}
\newcommand{\br}{\,\overline{\phantom{n}}\,} 

\newcommand{\dbf}{\itshape\bfseries}   

\newcommand{\ord}{\operatorname{ord}}

\newcommand{\End}{\operatorname{End}}

\newcommand{\rk}{\operatorname{rk}}

\newcommand{\nriso}{\cong_{\text{\rm nr}}}  

\newcommand{\Z}{{\mathbb Z}}
\newcommand{\Q}{{\mathbb Q}}

\newcommand{\Tc}{\operatorname{T_{cr}}}

\newcommand{\R}{\operatorname{R}}

\newcommand{\calP}{\mathcal{P}}
\newcommand{\calC}{\mathcal{C}}
\newcommand{\calA}{\mathcal{A}}
\newcommand{\calB}{\mathcal{B}}
\newcommand{\calG}{\mathcal{G}}

\newcommand{\calS}{\mathcal{S}}
\newcommand{\ptn}{\operatorname{ptn}}
\newcommand{\olP}{\overline{P}}

\newcounter{noteno}

\begin{document}
\title
{Indecomposable decompositions of
  torsion--free abelian groups}  

\author {Adolf Mader}
\address {Department of Mathematics \\
University of Hawaii at Manoa \\
2565 McCarthy Mall, Honolulu, HI 96922, USA}
\email{adolf@math.hawaii.edu}
  
\author {Phill Schultz}
\address {School of Mathematics and Statistics\\
The University of Western Australia \\ Nedlands\\
 Australia,  6009}
\email {phill.schultz@uwa.edu.au}
 


\begin{abstract} An indecomposable decomposition of a torsion-free
  abelian group $G$ of rank $n$ is a decomposition $ G = A_1 \oplus
  \cdots \oplus A_t $ where each $A_i$ is indecomposable of rank
  $r_i$, so that $\sum_i r_i = n$ is a partition of $n$. The group $G$
  may have indecomposable decompositions that result in different
  partitions of $n$. We address the problem of characterizing those
  sets of partitions of $n$ which can arise from indecomposable
  decompositions of a torsion-free abelian group.
\end{abstract} 

\maketitle


\section{Introduction} 

All the groups considered in this paper belong to the class $\calG$ of
finite rank torsion--free abelian groups; we shall refer to them
simply as \lq groups\rq. Any group $G \in \calG$ is viewed as an
additive subgroup of a finite dimensional $\Q$--vector space and its
{\dbf rank}, $\rk(G)$, is the dimension of the subspace $\Q G$ that $G$
generates. If a group $G$ of rank $n$ has an indecomposable
decomposition $G = A_1 \oplus \cdots \oplus A_t$ where $A_i$ is
indecomposable of rank $r_i$, then $\sum_i r_i = n$ so $P=(r_1,\ldots,
r_t)$ is an unordered partition of $n$ into $t$ parts, and we say that $G$
{\dbf realizes} $P$.   

\medskip A group $G$ may realize several different partitions and we denote by
$\calP(G)$ the set of all partitions that are realized by $G$.
Determining $\calP(G)$ for a given group $G$ and characterizing
realizable sets of partitions have received a great deal of attention
in the literature. For example, \cite[Section 90]{F2} displays groups
realizing $(1,3)$ and $(2,2)$, groups realizing partitions into $2$
and $t$ parts for each $2 \leq t < n$, and groups realizing every
partition of $n$ into $t$ parts (Corner's Theorem \ref{Corner} below).

\medskip
On the other hand, there are classes of groups, such as completely
decomposable groups, direct sums of indecomposable fully invariant
subgroups, and almost completely decomposable groups with primary
regulating index (see \cite[Theorem 3.5]{FS96}) that realize a unique
partition. Furthermore, some sets of partitions are incompatible, in
the sense that no group can realize them.  For example, we showed in
\cite[Theorem 3.5]{MS17} (\thmref{main decomposition}) that no group
can realize both $(k, 1^{n-k})$ and $(m, 1^{n-m})$ if $k \ne m$, where
$(i,1^{j})$ stands for the partition $(i,1,\ldots,1)$ with $j$ copies
of~$1$.  
\medskip
In this paper, we consider three problems:
\begin{enumerate}
\item Characterize families of partitions that can be realized by some
  group.
\item For a given family $\calP$ of partitions of $n$ that can be
  realized, determine the groups $G$ that realize $\calP$.
\item For a given class $\calA$ of groups, determine the families
  $\calP(G)$ of partitions realized by $G \in \calA$.
\end{enumerate}
\smallskip

For the most part, we adopt the notation of \cite{F2}.  The
definitions and properties of almost completely decomposable, rigid
and block--rigid cyclic regulator quotient groups can be found in
\cite{MA00}. The properties we need are described in Section
\ref{BRCRQ}.
\medskip

This paper contains the following results. By $\calB$ we denote
the class of all block--rigid almost completely decomposable groups
with cyclic regulator quotient.

\begin{enumerate}
\item The ``Hook condition'' is shown to be necessary for the sets of
  partitions $\calP(X)$ of groups $X \in \calB$ (\propref{P(G)
    hooked}(1)). 
\item A useful ``standard'' description of $\calB$--groups is established
  (\lemref{integral coordinates}).
\item The workhorse \lemref{sums of rigid crq groups} describes the
  direct sum of two $\calB$--groups given in standard description.
\item General obstructions to realization are established
  (\lemref{indecomposable fully invariant summand}, \lemref{completely
    decomposable tau homogeneous fully invariant summand},
  \lemref{bound for ranks}).
\item The maximal sets of partitions of $n$ that satisfy the Hook
  condition are described and denoted by $\calS(n,k)$ (\propref{unique
    hook}).  The (few) $\calS(n,k)$ that can be realized by groups in
  $\calB$ and the (many) that cannot so realized are determined
  (\thmref{complete answer}).
\end{enumerate} 

\section{Background}\label{BRCRQ}

The decomposition properties of torsion--free groups of finite rank in
general are extremely complex (see \cite{BlM94} and \cite{F2}). For
example, A. L. S. Corner \cite{Cor} proved an intriguing theorem on
\lq pathological\rq\ decompositions:

\begin{thm}\label{Corner} {\rm (Corner's Theorem)} Given integers 
  $n \geq k \geq 1$, there exists a group $G$ that realizes every
  partition of $n$ into $k$ parts.
\end{thm}

Specializations and simplifications are needed in order to have
any hope for results.

Our essential specialization consists in passing to the subclass of
almost completely decomposable groups, first defined and
systematically investigated by Lady (\cite{LA74}), and further to the
much more special but still interesting and rich subclass $\calB$ of
block rigid groups with cyclic regulator quotient. To our knowledge
all published examples of groups with ``pathological'' decompositions
are almost completely decomposable groups. The essential
simplification occurs by passing from isomorphism to Lady's
near--isomorphism (\cite{LA75}, \cite[Chapter 9]{MA00}). A theorem by
David Arnold (\cite[Corollary~12.9]{AR82},
\cite[Theorem~10.2.5]{MA00}) says that passing from isomorphism to the
coarser equivalence relation of Lady does not mean a loss of
generality for our line of questioning: if $G$ and $G'$ are nearly
isomorphic groups, then $\calP(G) = \calP(G')$. 
\medskip

  A completely decomposable group $A$ is the direct sum of
  rank--$1$ groups. We always use its \lq homogeneous decomposition\rq\ 
  $A = \bigoplus_{\rho \in \Tc(A)} A_\rho$ where $A_\rho$ is a
  non--zero direct sum of rank--$1$ groups all of type $\rho$.
  If so, the index set $\Tc(A)$ is the {\dbf critical typeset} of
  $A$.  An {\dbf almost completely decomposable group} is a group
containing a completely decomposable subgroup of finite index. Let $X$
be an almost completely decomposable group. Then $X$ contains
completely decomposable subgroups of least index in $X$, and these are
Lady's {\dbf regulating subgroups.} The intersection of all regulating
subgroups is a fully invariant completely decomposable subgroup of
$X$, Burkhardt's {\dbf regulator $\R(X)$ of $X$}. Let $\R(X) =
\bigoplus_{\rho \in \Tc(\R(X))} R_\rho$ be the homogeneous decomposition
of $R$. An almost completely decomposable group $X$ is {\dbf
  block--rigid} if $\Tc(\R(X))$ is an antichain and   {\dbf rigid} if
$\Tc(X)$ is an antichain and for all $\rho \in \Tc(X)$, $\rk R_\rho =
1$. We denote by $\calB$ the class of all block--rigid groups for
which $X/\R(X)$ is cyclic.
\smallskip

Let $G$ and $H$ be groups. We say that $G$ is {\dbf nearly
  isomorphic} to $H$, denoted $G \nriso H$, if there is a group $K$
such that $G \oplus K \cong H \oplus K$.  The properties of near
isomorphism are described in \cite[Section 12]{AR82} and in
\cite[Chapter 10]{MA00}. In our context, the most important such
property is \cite[Corollary 12.9]{AR82}, \cite[Theorem 10.2.5]{MA00}:
for any groups $G$ and $H$, if $G \nriso H$ and $G = G_1 \oplus \cdots
\oplus G_t$, then $H = H_1 \oplus \cdots \oplus H_t$ with $H_i \nriso
G_i$, in particular $\calP(G)=\calP(H)$.

E. Blagoveshchenskaya proved an interesting obstruction to realization
(\cite{BlM94}, \cite[Chapter~13]{MA00}) which we quote as Theorem
\ref{Blago2}.

\begin{thm}\label{Blago2} \rm{(\cite{BlM94})} Let $(r_1, \ldots,
  r_s)$ and $(r'_1, \ldots, r'_t)$ be two partitions of $n$. Then
  there exists a block--rigid crq--group realizing $(r_1, \ldots,
  r_s)$ and $(r'_1, \ldots, r'_t)$ if and only if $r_i \leq n-t+1$ for
  all $i = 1, \ldots, s$ and $r'_j \leq n-s+1$ for all $j = 1, \ldots,
  t$.
\end{thm}

Another obstruction to realization is implicit in the ``Main Decomposition
Theorem'' below. A group is {\dbf clipped} if it does not have a
rank--one summand. Clearly any group in $\calG$ has a completely
decomposable direct summand of maximal rank (that may be $0$) and the
complementary summand is necessarily clipped.  

\begin{thm}\label{main decomposition} {\rm (\cite[Theorem~2.5]{MS17})}
  Let $G \in \calG$ and let $G = G_{cd} \oplus G_{cl}$ be a
  decomposition of $G$ with completely decomposable summand $G_{cd}$
  and clipped summand $G_{cl}$. Then $G_{cd}$ is unique up to
  isomorphism and $G_{cl}$ is unique up to near--isomorphism, i.e., if
  also $G = G_{cd}' \oplus G_{cl}'$ is such that $G_{cd}'$ is
  completely decomposable and $G_{cl}'$ is clipped, then $G_{cd} \cong
  G_{cd}'$ and $G_{cl} \nriso G_{cl}'$.
\end{thm} 

\begin{cor}\label{hook}
 If $G$ realizes $(r,1^{n-r})$ and $(r',1^{n-r'})$, then $r=r'$. 
\end{cor}
  
\section{Hooked Partitions}

Recall that $\calP(G)$ denotes the set of all partitions $P = (r_1,
\ldots, r_t)$ of $\rk(G)$ for which there is an indecomposable
decomposition $G = G_1 \oplus \cdots \oplus G_t$ such that for all $i
\in \{1, \dots, t\}$, $r_i = \rk(G_i)$. There are two extreme values
of evident interest: Let $r^G$ be the largest rank of an
indecomposable summand of $G$, and let $t^G$ be the largest number of
summands in an indecomposable decomposition of $G$. Intuitively, if
$r^G$ is large, then $t^G$ should be small and conversely. For
example, if $r^G = \rk(G)$, then $t^G = 1$, $G$ is indecomposable and
$\calP(G) = \{(\rk(G))\}$; if $t^G = \rk(G)$, then $r^G = 1$, $G$ is
completely decomposable and $\calP(G) = \{(1, \ldots, 1)\}$. Thus if
$G$ is indecomposable or completely decomposable, then $r^G + t^G =
\rk(G)+1$. 

\medskip In view of Corner's Theorem denote by $\calC(n,k)$ the set
of all partitions of $n$ into $k$ parts.
The following proposition describes restrictions on realizable
partitions.

\begin{prop}\label{P(G) hooked} \hfill
\begin{enumerate}
\item For any $G \in \calB$, it is true that $r^G + t^G \leq \rk(G)+1$.
\item For any $G \in \calG$ with $\calP(G) = \calC(n,k)$ it is true
  that $r^G + t^G = n+1$.
\end{enumerate} 
\end{prop}

\begin{proof} (1) Let $G$ be a $\calB$--group of rank $n$ and suppose
  that $G = C \oplus D = B_1 \oplus \cdots \oplus B_t$ are
  decompositions of $G$ with $C$ indecomposable, $r^G = \rk C$ and all
  $B_i$ indecomposable with $t = t^G$. According to \thmref{Blago2}
  $r^G \leq n-t^G+1$ or $r^G + t^G \leq n+1$.

  (2) For a group $G$ with $\calP(G) = \calC(n,k)$ we have $t^G = k$
  and as $(n-k+1,1^{k-1}) \in \calC(n,k)$ we have $r^G = n-k+1$, hence
  $r^G + t^G = n+1$.
\end{proof}

Let $\calP(n)$ be the set of all partitions of $n$ and let $\calP
\subseteq \calP(n)$, so $\calP$ is a family of partitions of $n$. A
{\dbf hook} in $\calP(n)$ is a partition $(k,1^{n-k})$ where $k \in
\{1,\ldots,n\}$. Let $r^\calP$ be the largest entry of any partition
in $\calP$ and let $t^\calP$ be the length of the longest partition in
$\calP$. We say that {\dbf $\calP$ satisfies the Hook Condition} or
simply that {\dbf $\calP$ is hooked} if $r^\calP + t^\calP \leq n+1$.

\medskip By \propref{P(G) hooked} the family $\calP(G)$ is hooked if $G \in
\calB$ or if $\calP(G)=\calC(n,k)$ for some $k$.

\begin{qu}\label{hooked?} Is $\calP(G)$ hooked for every $G \in \calG$?
\end{qu}

We conjecture, but so far have no proof that the Hook Condition is
indeed necessary for $\calP(G)$.   
On the other hand, we show below that the Hook Condition is far from
sufficient to characterize $\calP(G)$.

\medskip Let $P \in \calP(n)$, so $P = (r_1, r_2, \cdots, r_t)$, $r_i \geq 1$,
$\sum r_i = n$. Then $\olP$ denotes the largest term of the partition
and $|P|=t$ is its length. 

\medskip For $n \geq 1$ and $1 \leq k \leq n$ we define \[\calS(n,k) = \{P \in
\calP(n) \mid \olP \leq k \text{ and } k+|P| \leq n+1\}.\] The sets
$S(n,k)$ are well known in Partition Theory under the name of
\lq restricted partitions\rq.  They are studied, enumerated and
tabulated for example in \cite[Section 14.7]{A71}.

\begin{prop} \label{unique hook} \hfill
\begin{enumerate}
\item The family of partitions $\calS(n,k)$ is hooked, contains a
  unique hook $(k,\,1^{n-k})$ and is maximal with respect to inclusion
  among the hooked families of partitions of $n$.
\item If a group $G$ realizes $\calS(n,k)$ and $\calP(G)$ is hooked,
  then $\calP(G) = \calS(n,k)$.
\item If a group $G$ realizes $\calC(n,k)$ and $\calP(G)$ is hooked,
  then $\calP(G) = \calC(n,k)$.
  \end{enumerate}
\end{prop}

\begin{proof} (1) \cite[Section 14.7]{A71}.

  (2) By hypothesis $\calS(n,k) \subseteq \calP(G)$. As
  $\calP(G)$ is hooked and $\calS(n,k)$ is a maximal hooked family of
  partitions we must have equality.

  (3) Suppose that $\calC(n,k) \subseteq \calP(G)$ and $\calP(G)$ is
  hooked. Then $t^G \geq k$ and as $(n-k+1,1^{k-1}) \in \calC(n,k)
  \subseteq \calP(G)$, further $r^G \geq n-k+1$. Now $\calP(G)$ being
  hooked we get $n+1 \geq r^G + t^G \geq (n-k+1) + k = n+1$ and
  therefore $r^G + t^G = n+1$. Furthermore $r^G = n-k+1$ and $t^G =
  k$ because either $r^G > n-k+1$ or $t^G > k$ contradicts the Hook
  condition. 
\end{proof}

\begin{qu}\label{indecomposable summand of largest rank} Suppose $G
  \in \calG$  realizes $(k, 1^{n-k})$; then is
  $r^G = k$ and/or $t^G = n-k+1$?
\end{qu} 

\medskip The answer is yes for a group $G$ with $\calP(G) = \calC(n,k)$ and any
group realizing some $\calS(n,k)$. 
  On the other hand, it is not known if there  are   groups realizing, for example, 
  both $(6,1,1)$ and $(2,2,2,2)$.  

\section{Fully invariant direct summands}

Recall that $\calP(X)$ denotes that set of partitions of $\rk(X)$ that
can be realized by indecomposable decompositions of $X$. Recall that
$\calS(n,k) = \{P \in \calP(n) \mid \olP \leq k \text{ and } k+|P| \leq
n+1\}$. For an indecomposable decomposition $X = \bigoplus_{\rho}
X_\rho$ let $\ptn(\bigoplus_{\rho} X_\rho) = (\ldots, \rk(X_\rho),
\ldots) \in \calP(\rk X)$.

\begin{prop}\label{direct sum of fully invariant subgroups} Suppose
  that $X = Y \oplus Z$ and that both $X$ and $Y$ are fully invariant
  in $X$. Then the following statements are true.
\begin{enumerate}
\item Suppose that $X = \bigoplus_{\omega \in \Omega} X_\omega$ is an
  indecomposable decomposition of $X$. Then there is a disjoint union
  $\Omega = \Omega_Y \cup \Omega_Z$ such that $Y = \bigoplus_{\omega \in
    \Omega_Y} X_\omega$ and $Z = \bigoplus_{\omega \in \Omega_Z} X_\omega$.
\item $\calP(X) = \calP(Y) \times \calP(Z)$.
\end{enumerate} 
\end{prop} 

\begin{proof} (1) As $Y$ and $Z$ are fully invariant in $X$ it follows
  that $Y = \bigoplus_{\omega} (Y \cap X_\omega)$ and $Z =
  \bigoplus_{\omega} (Z \cap X_\omega)$. It is easily checked that
  $X_\omega = Y \cap X_\omega \oplus Z \cap X_\omega$. As $X_\omega$
  is indecomposable it follows that either $Y \cap X_\omega =0$ and $Z
  \cap X_\omega = X_\omega$ or $Y \cap X_\omega = X_\omega$ and $Z
  \cap X_\omega = 0$. Let $\Omega_Y = \{\omega \in \Omega \mid Y \cap
  X_\omega = X_\omega\}$ and $\Omega_Z = \{\omega \in \Omega \mid Z
  \cap X_\omega = X_\omega\}$. Then $Y = \bigoplus_{\omega \in
    \Omega_Y} X_\omega$ and $Z = \bigoplus_{\omega \in \Omega_Z}
  X_\omega$ as claimed.

  (2) It follows from (1) that $\ptn(\bigoplus_{\omega \in \Omega}
  X_\omega) = \ptn(\bigoplus_{\omega \in \Omega_Y} X_\omega) \times
  \ptn(\bigoplus_{\omega \in \Omega_Z} X_\omega)$, so $\calP(X)
  \subseteq \calP(Y) \times \calP(Z)$. Conversely, if $Y =
  \bigoplus_{\omega \in \Omega_Y} Y_\omega$ and $Z = \bigoplus_{\omega \in
    \Omega_Z} Z_\omega$ are indecomposable decompositions of $Y$ and
  $Z$, then $X = (\bigoplus_{\omega \in \Omega_Y} Y_\omega) \oplus
  (\bigoplus_{\omega \in \Omega_Z} Z_\omega)$ is an indecomposable
  decomposition of $X$ and \newline
  $\ptn((\bigoplus_{\omega} Y_\omega) \oplus
  (\bigoplus_{\omega} Z_\omega)) = \ptn(\bigoplus_{\omega} Y_\omega)
  \times \ptn(\bigoplus_{\omega} Z_\omega)$, so $\calP(Y) \times
  \calP(Z) \subseteq \calP(X)$.
\end{proof}

\lemref{indecomposable fully invariant summand} and \lemref{completely
  decomposable tau homogeneous fully invariant summand} exhibit
obstructions to realization. Assuming that certain decompositions are
present, they say that certain other decompositions and certain
critical typesets are impossible (see \lemref{pinning down critical
  types}).

\begin{lem}\label{indecomposable fully invariant summand} Let $G \in
  \calG$ and suppose that $H$ is a non--zero indecomposable fully
  invariant summand of $G$. Then $H$ appears as a summand in any
  indecomposable decomposition of $G$.
\end{lem}

\begin{proof} Let $G = G_1 \oplus \cdots \oplus G_t$ be an
  indecomposable decomposition of $G$. As $H$ is fully invariant in
  $G$ we have $H = H \cap G_1 \oplus \cdots \oplus H \cap G_t$. As $H$
  is indecomposable, without loss of generality $H = H \cap G_1$, and
  moreover $G_1 = H \oplus H'$. But $G_1$ is also indecomposable and
  so it follows that $G_1 = H$.
\end{proof}

\begin{lem}\label{completely decomposable tau homogeneous fully
    invariant summand} Let $G \in \calG$ and suppose that $H$ is a
  non--zero completely decomposable $\tau$--homogeneous fully 
  invariant summand of $G$. Then every indecomposable decomposition of
  $G$ contains a rank--$1$ summand of type $\tau$.
\end{lem}

\begin{proof} Let $G = G_1 \oplus \cdots \oplus G_t$ be an
  indecomposable decomposition of $G$. As $H$ is fully invariant in
  $G$ we have $H = H \cap G_1 \oplus \cdots \oplus H \cap G_t$.
  Without loss of generality $H \cap G_1 \neq 0$, and moreover $G_1 =
  (H \cap G_1) \oplus H'$. But $G_1$ is indecomposable and so it
  follows that $G_1 = H \cap G_1$. This means that $G_1 \subseteq H$
  and furthermore $G_1$ is a direct summand of the $\tau$--homogeneous
  completely decomposable group $H$, hence is itself
  $\tau$--homogeneous completely decomposable. Being indecomposable,
  $G_1$ has rank $1$ and type $\tau$.
\end{proof}

\section{$\calB$--groups}

This section lists essential details on block--rigid groups $X$ with
cyclic regulator quotient, i.e., groups in $\calB$, also referred to
as \lq $\calB$--groups \rq or \lq block--rigid crq--groups\rq.  
  We will use greatest common divisors $\gcd^A(e,a)$ where $e$ is a
  positive integer and $a$ is an element of the torsion-free group
  $A$. Greatest common divisors are explained in
  \cite[Section~11.2]{MA00} in great generality and they have the
  usual properties. In our special case $\gcd^A(e,a) = d$ when $m =
  \ord(a+eA)$ and $e = m d$.

A $\calB$--group $X$ has a representation
\begin{equation}\label{eq: representation BRCRQ}
  X = A + \Z e\invs a, \quad A = \bigoplus_{\rho \in \Tc(A)}
  A_\rho, \quad a = \sum_{\rho \in \Tc(A)} a_\rho, \quad a_\rho \in
  A_\rho
\end{equation}
with the following properties
\begin{equation}
  [X:A] = e \text{ if and only if } {\gcd}^A(e,a)=1,
\end{equation} 
\begin{equation}
  A = \R(X) \text{ if and only if for all }
   \tau \in \Tc(A),\  
  A(\tau) = A_\tau = X(\tau). \text{ (\cite[4.5.1]{MA00})}  
\label{eq: regulator criterion 0} 
\end{equation}

There is a more practical version of the Regulator Criterion
\eqref{eq: regulator criterion 0}.

\begin{lem}\label{regulator criterion 1} {\rm (Regulator Criterion)} 
  Assume that $\gcd^A(e,a)=1$ in \eqref{eq: representation BRCRQ}.
  Then $A = \R(X)$ if and only if for all $\tau \in \Tc(A),\ 
  \gcd^A(e,\sum_{\rho \neq \tau} a_\rho) = 1$.
\end{lem}

\begin{proof} Set $B := A_\tau$, $C = \bigoplus_{\rho \neq \tau}
  A_\rho$ and $c = \sum_{\rho \neq \tau} a_\rho$. Then $A = B \oplus
  C$. By the Purification Lemma (\cite[3.1.2,11.4.1]{MA00}), $B_*^X =
  B + \Z e_\tau\invs a_\tau$ where $e_\tau = \gcd^A(e,c)$. 
  
  
   So $e_\tau
  \mid e$ and $c = e_\tau c'$ for some $c' \in A$. Suppose that $A =
  \R(X)$. Then $B = X(\tau)$ is pure in $X$, hence $e_\tau\invs a_\tau
  \in A_\tau \subseteq A$ and $e_\tau(e_\tau\invs a_\tau + c') = a_\tau
  + c = a$. By assumption $\gcd^A(e,a) = 1$, so $e_\tau = 1$.
  Conversely, if $\forall\, \tau \in \Tc(A): e_\tau = 1$, then
  $\forall\, \tau \in \Tc(A): A_\tau = \left(A_\tau\right)_*^X =
  X(\tau)$ and $A = \R(X)$.
\end{proof}

We next restate well--known invariants and the classification of
$\calB$--groups. 

\begin{defn}\label{defn: m} {\rm \cite[Definition~12.6.2]{MA00}} Let
  $X$ be as in \eqref{eq: representation BRCRQ}. Let $e$ be any
  positive integer such that $eX \subset A$. Let $\br \: A \ra A/eA =
  \ol{A}$ be the natural epimorphism. For any type $\tau$, set
  \[
  \mu_\tau(X) = \ord(\ol{a_\tau}) = \ord(a_\tau + e A).
  \]
\end{defn}

The values $\mu_\tau(X)$, do not depend on the choices of $a$ and $e$
involved in their definition (\cite[Lemma~12.6.3]{MA00}).

The classification theorem for block--rigid crq--groups is easy to
state. 

\begin{thm}\label{near iso criterion} {\rm
    \cite[Theorem~12.6.5]{MA00}\ (Near--Classification)}
    Let $X$ and $Y$ be block--rigid crq--groups. Then $X \nriso Y$ if
    and only if $\R(X) \cong \R(Y)$ and $\mu_\tau(X) = \mu_\tau(Y)$
    for all types $\tau$.
\qed\end{thm} 

\thmref{product formula} says in particular that direct summands of
$\calB$--groups are again $\calB$--groups and relates the invariants
of group and summands. 

\begin{thm}\label{product formula} {\rm \cite[Theorem~13.1.1]{MA00}
  (Product Formula)} 
  Let $X$ be a block--rigid crq--group with regulator $R$. Suppose $X
  = X_1 \oplus X_2$ is a direct decomposition of $X$. Then each $X_i$
  is a block--rigid crq--group with regulator $R_i = X_i \cap R$, and 
  $$
  \frac{X}{R} \cong \frac{X_1}{R_1} \oplus \frac{X_2}{R_2}
  $$
  is a decomposition of the cyclic group $X/R$. The groups
  $X_i/R_i$ are cyclic and have relatively prime orders.
  Furthermore, 
  $$
  \mu_\tau(X) = \mu_\tau(X_1)\cdot \mu_\tau(X_2).
  $$
\qed\end{thm}

\lemref{direct sums in B} deals with direct sums in $\calB$ and their
regulators.  This lemma justifies the direct sums that we use in
  the later sections.

\begin{lem}\label{direct sums in B} Let $X_1, X_2 \in \calB$ with
  regulators $R_1, R_2$ respectively. Assume that $\Tc(X_1) \cup
  \Tc(X_2)$ is an antichain and $\gcd([X_1:R_1],[X_2:R_2])=1$. Then $X
  = X_1 \oplus X_2 \in \calB$ and $\R(X) = R_1 \oplus R_2$.
\end{lem} 

\begin{proof} Set $R = R_1 \oplus R_2$. Then $\frac{X}{R} \cong
  \frac{X_1}{R_1} \oplus \frac{X_2}{R_2}$ is a finite cyclic group of
  order $[X_1:R_1][X_2:R_2]$ and $\Tc(X) = \Tc(R) = \Tc(R_1) \cup
  \Tc(R_2) = \Tc(X_1) \cup \Tc(X_2)$ is an antichain. For any $\tau
  \in \Tc(X)$, we have $R(\tau) = R_1(\tau) \oplus R_2(\tau) =
  X_1(\tau) \oplus X_2(\tau) = X(\tau)$, so $\R(X) = R$ by
  \eqref{eq: regulator criterion 0}. So $X \in \calB$.
\end{proof} 

A criterion for the absence of rank--one direct summands is needed.

\begin{prop}\label{clipped iff} {\rm \cite[13.1.5]{MA00}} A
  block--rigid crq--group $X$ is clipped if and only if $X$ is rigid
  and $\mu_\tau(X) > 1$ for all $\tau \in \Tc(X)$.  
\qed\end{prop} 

The indecomposability of a $\calB$--group is easily recognized by
associating a certain graph with the group.  

\begin{defn}
  \textup{(Blagoveshchenskaya)} Let $X$ be a rigid $\calB$--group. The
  {\dbf frame} of $X$ is a graph whose vertices are the elements of
  $\Tc(X)$ and two vertices $\sigma$ and $\tau$ are joined by an edge
  if and only if $\gcd(\mu_\sigma(X),\mu_\tau(X)) > 1$.
\end{defn}

\begin{prop}\label{indecomposable via frame}
  {\rm \cite[Corollary~13.1.10]{MA00}} A $\calB$--group is
  indecomposable if and only if it is rigid and its frame is
  connected.
\qed\end{prop}

Recall  that any completely decomposable group $A$
of rank~$r$ can be written in the form $A = \bigoplus_{1 \le i \le r}
A_i = \bigoplus_{1 \le i \le r} \tau_i v_i$ where the $A_i$ are
rank~$1$ summands of $A$, $0 \neq v_i \in A_i$ and $\tau_i = \{x \in
\Q \mid x v_i \in A_i\}$. The $\tau_i$ are {\dbf rational groups},
i.e. additive subgroups of $\Q$ containing $1$. We identify a rational
group $\tau$ with its type, and so speak of the rational group $\tau$
of type $\tau$.  The set $\{v_1, \ldots, v_r\}$ is called a {\dbf
  decomposition basis} of $A$.

The following \lemref{integral coordinates} is restricted to rigid
groups because this simplifies the notation considerably and is all
that is used later.
\smallskip

\begin{lem}\label{integral coordinates} Let $X$ be a rigid
  $\calB$--group and let $R := \R(X) = \bigoplus_{\rho \in \Tc(X)}
  R_\rho$. Then $X$ has a description $X = R + \frac{\Z}{e} \sum_{\rho
    \in \Tc(X)} \alpha_\rho v_\rho$ where 
\begin{enumerate} 
\item $\{v_\rho \mid \rho \in \Tc(R)\}$ is a decomposition basis of
  $R$, so $R = \bigoplus_{\rho \in \Tc(X)} \rho v_\rho$,
\item $\forall\, \tau \in \Tc(X) : \gcd^R(e,\sum_{\rho \neq \tau}
  \alpha_\rho v_\rho) = 1$, 
\item $e = [X : R]$,
\item $\alpha_\rho$ an integral factor of $e$, and
\item $\forall\, \tau \in \Tc(X): \mu_\tau(X) = \frac{e}{\alpha_\tau}$.
\end{enumerate} 

  If $\rk(X) = 2$, $R = \R(X) = \tau_1 v_1 \oplus \tau_2 v_2$ and $[X
  : R] = e$, then we may assume the representation $X = A + \Z
  e\invs(v_1 + v_2)$ with $\gcd^R(e,v_i) = 1$. 
\end{lem} 

\begin{proof} Certainly, $X = R + \Z \frac{1}{e} \sum_{\rho \in
    \Tc(X)} u_\rho$, where $u_\rho \in R_\rho$ and (Regulator
  Criterion) for all $\tau \in \Tc(X)$, $\gcd^R(e,\sum_{\rho \neq
    \tau} u_\rho) = 1$. Let $\rho \in \Tc(X)$ and $\alpha_\rho =
  \gcd^R(e,u_\rho)$. Then
  \[\textstyle
  e = \alpha_\rho \beta_\rho, \quad u_\rho = \alpha_\rho v_\rho,
  \quad\text{and}\quad \gcd^R(\beta_\rho, v_\rho) = 1.
  \]
  Let $\rho = \{x \in \Q \mid x v_\rho \in R_\rho\}$. Then $\rho$ is a
  rational group of type $\rho$ and $R_\rho = \rho v_\rho$. 

  It remains to show that $\mu := \mu_\tau(X) = \ord(\alpha_\tau
  v_\tau + eR) = \frac{e}{\alpha_\tau}$. Clearly
  $\frac{e}{\alpha_\tau} (\alpha_\tau v_\tau + e R) = 0$, so $\mu \mid
  \frac{e}{\alpha_\tau}$. On the other hand $\mu = \ord(\alpha_\tau
  v_\tau + e R) = \ord(\alpha_\tau v_\tau + \alpha_\tau \beta_\tau R)
  = \ord(v_\tau + \beta_\tau R_\tau) = \ord(v_\tau + \beta_\tau \tau
  v_\tau)$. Hence $\mu v_\tau = \beta_\tau \frac{a}{b} v_\tau$ where
  $\frac{a}{b}$ is a reduced fraction in the rational group $\tau$. It
  follows that $b \mu = a \beta_\tau$ and hence $\beta_\tau = b
  b'$. Furthermore as $\frac{a}{b} v_\tau \in \tau v_\tau = R_\tau$ we
  also have $\frac{1}{b} v_\tau \in R_\tau$ and $v_\tau = b
  \left(\frac{1}{b} v_\tau \right)$. Thus $b$ is a common factor of
  $\beta_\tau$ and $v_\tau$. But $\gcd^R(\beta_\tau, v_\tau) = 1$, so
  $b=1$ and $\frac{e}{\alpha_\tau} = \beta_\tau \mid
  \mu_\tau$. Together with the previous observation that $\mu \mid
  \frac{e}{\alpha_\beta}$ we have the desired equality. 

  Consider $X = R + \Z e\invs(\alpha_1 v_1 + \alpha_2 v_2)$ where $R =
  \R(X) = \tau_1 v_1 \oplus \tau_2 v_2$ and $\alpha_i \mid e$. The
  Regulator Criterion requires that $\gcd^R(e, \alpha_i v_i) = 1$ this
  forces $\alpha_i = 1$. But then $\gcd^R(e, v_i) = 1$.
\end{proof}

Given a rigid $\calB$--group we can assume without loss of generality
that it is in the ``standard form'' described in \lemref{integral
  coordinates}. On the other hand, in constructing rigid
$\calB$--groups, in particular indecomposable $\calB$--groups, we must
make judicious choices in order to achieve the standard form and the
easy way of computing invariants.

Let $e$ be a positive integer. A type $\tau$ is {\dbf $e$--free} if
for all prime divisors $p$ of $e$, $p \tau \neq \tau$. It is easily
seen that a rank--$1$ group $A$ of $e$--free type $\tau$ has a
representation $A = \tau v$ such that $\gcd^A(e,v) = 1$. Consequently,
a rigid completely decomposable group $A$ all of whose critical types
are $e$--free has a description $A = \bigoplus_{\rho \in \Tc(A)} \rho
v_\rho$ with $\gcd^A(e,v_\rho) = 1$.   In this case $\{v_\rho :
  \rho \in \Tc(A)\}$ is called an {\dbf $e$--basis of $A$.}

It is easy to construct indecomposable $\calB$--groups with evident
invariants.

\begin{lem}\label{existence indecomposable} Let $e$ be a positive integer,
  and let $\tau_1, \ldots \tau_r$ be pairwise incomparable types such
  that $\forall\, i: \tau_i \text{ is $e$--free}$. Let $A =
  \bigoplus_{1 \le i \le r} \tau_i v_i$ such that $\forall\, i :
    \gcd^A(e,v_i)=1$. Then 
  \[\textstyle
  X = A + \frac{\Z}{e}(\alpha_1 v_1 + \cdots + \alpha_r v_r)
  \]
  with $\alpha_i$ integral factors of $e$ is an indecomposable
  $\calB$--group with $A = \R(X)$, $[X:A]=e$ and $\mu_{\tau_i}(X) =
  e/\alpha_i$ provided that the $\alpha_i$ are so chosen that
\begin{enumerate}
\item $\forall\, i : \gcd^A(e,\sum_{j \neq i} \alpha_j v_j) = 1$,
\item the frame of $X$ is connected. 
\end{enumerate}
\end{lem} 

\begin{proof} Evident by \lemref{regulator criterion 1} and
  \propref{indecomposable via frame}. 
\end{proof} 

For example, for $e = 3^2 \cdot 5 \cdot 7^3$ and $\tau_i$ $e$--free,
the group 
  \[\textstyle
  X = (\tau_1 v_1 + \tau_2 v_2 + \tau_3 v_3 + \tau_4 v_4 +
  \tau_5 v_5) + \frac{\Z}{e}(3^2 v_1 + 7 v_2 + 3 v_3 + 5 v_4 + v_5) 
  \] 
is an indecomposable $\calB$--group with invariants 
$\mu_1(X) = 5 \cdot 7^3, \mu_2(X) = 3^2 \cdot 5 \cdot 7^2,
\mu_3(X) = 3 \cdot 5 \cdot 7^3, \mu_4(X) = 3^2 \cdot 7^3,
\mu_5(X) = 3^2 \cdot 5 \cdot 7^3$.  

The indecomposable decompositions of rigid $\calB$--groups are
unique. Hence different indecomposable decompositions can only occur
in properly block--rigid groups. 

\begin{thm}\label{rigid unique} The endomorphism ring of a rigid
  $\calB$--group is commutative. Consequently, rigid $\calB$--groups
  have a unique indecomposable decomposition.
\end{thm} 

\begin{proof} Let $X$ be a rigid $\calB$--group. The restriction map
  imbeds $\End(X)$ in $\End(\R(X))$ and we will see shortly that
  $\End(\R(X))$ is commutative. Let $R := \R(X) = \bigoplus_{\rho
    \in \Tc(X)} R(\rho)$ where $\rk(R(\rho)) = 1$ and let $\varphi \in
  \End(R)$. Then $\varphi$ restricts to an endomorphism $\varphi_\rho$
  of $R(\rho)$ and hence is multiplication by a rational number. Hence
  $\varphi = (\ldots, \varphi_\rho, \ldots)$ and such maps commute
  with one another.

  In a group with commutative endomorphism ring every endomorphic
  image is a fully invariant subgroup. In particular, any direct
  summand is fully invariant. Let $X = \bigoplus_i X_i = \bigoplus_j
  Y_j$ be indecomposable decompositions.
  Let $\tau \in \Tc(X_1)$. Then there is $j$ such that $\tau \in
  \Tc(Y_j)$. Without loss of generality $j=1$ and for convenience
  write $X = Y_1 \oplus Y_1'$. As $X_1$ is fully invariant, $X_1 = (X_1
  \cap Y_1) \oplus (X_1 \cap Y_1')$.  As $X_1$ is indecomposable and
  $0 \neq X(\tau) = R(\tau) \subset X_1 \cap Y_1$ it follows that $X_1
  \subset Y_1$ and $X_1 \cap Y_1' = 0$. Then $X_1$ is a summand of
  $Y_1$ and since $Y_1$ is indecomposable, so $X_1 = Y_1$. The claim
  follows by induction on the number of summands $X_i$.
\end{proof} 

\begin{cor}\label{clipped so unique decompositions} If $X$ is a
  clipped $\calB$--group, then $X$ has a unique indecomposable
  decomposition. 
\end{cor}

\begin{proof} Let $X$ be a clipped $\calB$--group. By
  \propref{clipped iff} $X$ is rigid and by \thmref{rigid unique} $X$
  has a unique indecomposable decomposition. 
\end{proof}

\section{Tools and Obstructions}

We start with a consequence of realization.

\begin{lem}\label{bound for ranks} Let $X \in \calB$. Let
  $R = \R(X) = \bigoplus_{\rho \in \Tc(X)} R_\rho$. Suppose that $X =
  X_1 \oplus \cdots \oplus X_s$ is an indecomposable decomposition of
  $X$. Then $\forall\, \tau \in \Tc(X) : \rk(R_\tau) \leq s$. In
  particular, if $X$ realizes $(r_1,r_2)$, then $\forall\, \tau \in
  \Tc(X) : \rk(R_\tau) \leq 2$.
\end{lem} 

\begin{proof} Indecomposable $\calB$--groups are rigid. Let $\tau \in
  \Tc(X)$. Then $R_\tau = X(\tau) = X_1(\tau) \oplus \cdots \oplus
  X_s(\tau)$ with $\rk X_i(\tau) \leq 1$ (\ref{eq: regulator
    criterion 0}). Hence $\rk R_\tau = \rk X(\tau) \leq s$.
\end{proof} 

We use the following convenient shorthand that displays essential
features of indecomposable $\calB$--groups but not the whole story.

\medskip {\bf Notation. } The expression $[\tau_1\, \tau_2\, \ldots \, \tau_r]$
denotes an indecomposable $\calB$--group $X$ with critical types
$\tau_1, \ldots, \tau_r$ and rank $r$. It is tacitly assumed that $X$
is given in the form $X = R + \frac{\Z}{e} a$ where $R = \R(X) =
\bigoplus_{\rho \in \Tc(X)} \rho v_\rho$, $a \in R$, and $[X:R] = e$.

On occasion the regulator index $e$ is relevant to the discussion and
we encode the groups $X = R + \frac{\Z}{e} a$ by $[\tau_1\,\tau_2\,
\ldots\, \tau_r; e]$. In a direct sum $[\tau_1\,\tau_2\, \ldots; e_1]
\oplus [\sigma_1, \sigma_2, \ldots; e_2]$ it is tacitly assumed that   the poset of types 
$\{\tau_i, \sigma_j\}$ is an antichain, and the regulator indices
$e_1, e_2$ are relatively prime so that the direct sum is a group in
$\calB$ {\rm (\lemref{direct sums in B})}.
Of course, when the invariants $\mu_\tau(X)$ are needed, the explicit
representations \eqref{eq: representation BRCRQ} must be used.

\medskip If a group realizes particular partitions, then there are consequences
for the critical typeset of the group.

\begin{lem}\label{pinning down critical types} Suppose that $X$
  realizes $(k,1^{n-k})$ and $(r_1,r_2)$ where $r_1,r_2 \geq 2$.
  Then $\Tc(X) = \{\tau_1, \ldots, \tau_k\}$ and $X = [\tau_1 \,
  \cdots \, \tau_k] \oplus [\sigma_1] \oplus \cdots [\sigma_{n-k}]$
  where $\{\sigma_1, \ldots, \sigma_{n-k}\} \subseteq \{\tau_1,
  \ldots, \tau_k\}$.
\end{lem} 

\begin{proof} Let $\R(X) = \bigoplus_{\rho \in \Tc(X)} R_\rho$. As $X$
  realizes $(r_1,r_2)$, by \lemref{bound for ranks}, $\forall\, \rho
  \in \Tc(X): \rk(R_\rho) \leq 2$. As $r_1,r_2 \geq 2$ the group $X$
  has no fully invariant summand of rank $1$ (\lemref{completely
    decomposable tau homogeneous fully invariant summand}). As $X$ realizes
  $(k,1^{n-k})$, we have $X = [\tau_1 \, \cdots \, \tau_k] \oplus
  [\sigma_1] \oplus \cdots \oplus [\sigma_{n-k}]$. If $\sigma_i \notin
  \{\tau_1, \ldots, \tau_k\}$, then $[\sigma_i] \oplus \cdots$ is a
  fully invariant direct summand of $X$ and hence any indecomposable
  decomposition of $X$ contains a summand $[\sigma_i]$
  (\lemref{completely decomposable tau homogeneous fully invariant
    summand}), contradicting our earlier observation.
\end{proof} 

We now consider Corner's groups although as a consequence of
\lemref{bound for ranks} Corner's groups are of no help in trying to
realize $\calS(n,n-k+1)$. In particular, we show that Corner groups
are block--rigid crq--groups. \thmref{decompositions of Corner
  groups}(2) is not new (\cite{Cor}, \cite[Exercise~3, p.139]{F2}) but
it is easily derived from our results.

\begin{defn}\label{Corner groups} \rm Given
\begin{enumerate} 
\item integers $k,n$ with $1 \leq k \leq n$,
\item pairwise relatively prime integers $>1$, $q_1, \ldots, q_{n-k}$,
\item $e$--free pairwise incomparable types $\tau, \tau_1, \ldots,
  \tau_{n-k}$, 
\end{enumerate}
a {\dbf Corner group} $X$ has base group 
  \[\textstyle
  R = \tau u_1 \oplus \cdots \oplus \tau u_k \oplus \tau_1 v_1 \oplus
  \cdots \oplus \tau_{n-k} v_{n-k},
  \]
and is defined to be 
  \[\textstyle
  X = R + \frac{\Z}{q_1}(u_1+v_1) + \frac{\Z}{q_2}(u_1+v_2) + \cdots +
  \frac{\Z}{q_{n-k}}(u_1+v_{n-k}).
  \] 
  Hence $\rk(R_\tau) = k$ and for all $ i\in[n-k],\  \rk(R_{\tau_i}) = 1$.
\end{defn}

\begin{thm}\label{decompositions of Corner groups} Let $X$ be a
  Corner group as in \defref{Corner groups}. Then the following hold.
\begin{enumerate}
\item Setting $e := q_1 q_2 \cdots q_{n-k}$, and $e_i := e/q_i$, $R_0
  := \tau u_2 \oplus \cdots \oplus \tau u_k$ and $R_1 := \tau u_1 \oplus
    \tau_1 v_1 \oplus \cdots \oplus \tau_{n-k} v_{n-k}$, 
\begin{equation*}
  X = R_0 \oplus \left(R_1 +
    \frac{\Z}{e}((e_1 + \cdots + e_{n-k}) u_1 + e_2 v_2 + \cdots +
    e_{n-k} v_{n-k})\right) \in \calB
\end{equation*}
  realizing $(n-k+1,1^{k-1}) \in \calP(n)$.
\item $\calP(X) = \calC(n,k)$.
\end{enumerate} 
\end{thm} 

\begin{proof} (1) \allowbreak $\frac{1}{q_1}(u_1+v_1) +
  \frac{1}{q_2}(u_1+v_2) + \cdots + \frac{1}{q_{n-k}}(u_1+v_{n-k}) =
  \frac{1}{e}((e_1 + \cdots + e_{n-k}) u_1 + e_2 v_2 + \cdots +
  e_{n-k} v_{n-k})$. Hence $X$ contains $(\tau u_2 \oplus
  \cdots \oplus \tau u_k) \oplus$ \newline $\left((\tau u_1 \oplus
    \tau_1 v_1 \oplus \cdots \oplus \tau_{n-k} v_{n-k}) +
    \frac{\Z}{e}((e_1 + \cdots + e_{n-k}) u_1 + e_2 v_2 + \cdots +
    e_{n-k} v_{n-k})\right)$. Equality follows from \lemref{get
    summands back}.
 
  (2) By (1) a Corner group $X$ is a $\calB$--group hence $\calP(X)$
  is hooked. The claim follows from \propref{unique hook}(3).
\end{proof} 

 \smallskip

\lemref{sums of rigid crq groups} is basic. In its proof \lemref{get
  summands back} will be used.

\begin{lem}\label{get summands back} Let $R$ be a torsion-free group,
  $a,b \in R$, $e_1, e_2$ relatively prime integers, and $1 = u_1
  e_1 + u_2 e_2$. Then, in $\Q R$,   
\begin{enumerate} 
\item $\frac{1}{e_1} a + \frac{1}{e_2} b = 
  \frac{1}{e_1 e_2}(e_2 a + e_1 b)$, but also
\item $u_1 e_1 \left(\frac{1}{e_1 e_2}(e_2 a + e_1 b)\right) =
  \frac{1}{e_2}b + u_1 a - u_2 b$, and 
\item $u_2 e_2 \left(\frac{1}{e_1 e_2}(e_2 a + e_1 b)\right) =
  \frac{1}{e_1}a - u_1 a + u_2 b$
\end{enumerate}
  Hence $R + \frac{1}{e_1 e_2}(e_2 a + e_1 b) = 
  R + \frac{1}{e_1} a + \frac{1}{e_2} b$. 
\end{lem}

\begin{proof} Easy computation.
\end{proof} 

\begin{cor}\label{direct sum of two groups} For $i = 1,2$, let $X_i =
  R_i + \frac{\Z}{e_i} a_i \in \calB$ where $\Tc(X_1) \cup \Tc(X_2)$
  is an antichain, $e_1, e_2$ are relatively prime integers, $R_i =
  \R(X_i)$, and $a_i \in R_i$. Then $X_1 \oplus X_2 \in
  \calB$,
  \[\textstyle
  X_1 \oplus X_2 = (R_1 \oplus R_2) 
         + \frac{\Z}{e_1 e_2}(e_2 a_1 + e_1 a_2) \text{ and } 
  \R(X_1 \oplus X_2) = R_1 \oplus R_2.
  \]  
\end{cor} 

\begin{proof} \lemref{get summands back} says that $X_1 \oplus X_2 =
  (R_1 \oplus R_2) + \frac{\Z}{e_1 e_2}(e_2 a_1 + e_1 a_2)$. By
  hypothesis $\R(X_i) = R_i$, and, by \lemref{direct sums in B}, $\R(X_1
  \oplus X_2) = R_1 \oplus R_2$. 
\end{proof}

We now can freely add two $\calB$--groups by adding generators, the
description of the sum and its regulator being clear. We clarify in
all generality what the direct sum of two rigid $\calB$--groups looks
like. Note that the essential feature is the overlap in critical types
of the two groups that are added.

\begin{lem}\label{sums of rigid crq groups} Let $\tau_1, \ldots,
  \tau_d, \tau_{d+1}, \ldots, \tau_r, \tau_{r+1}, \ldots \tau_s$ be
  pairwise incomparable types, 
  $d \geq 1$, $r \geq d$, $s \geq r$
\begin{eqnarray*}
  A_1 &=& \tau_1 v_1 \oplus \tau_2 v_2 \oplus \cdots \oplus \tau_d v_d
  \oplus \tau_{d+1} v_{d+1} \oplus \cdots \oplus \tau_r v_r, \\
  A_2 &=& \tau_1 w_1 \oplus \tau_2 w_2 \oplus \cdots \oplus \tau_d w_d
  \oplus \tau_{r+1} v_{r+1} \oplus \cdots \oplus \tau_s v_s,
\end{eqnarray*}  
  and 
  \[
  X_1 = A_1 + \frac{\Z}{e_1}(\alpha_1 v_1 + \alpha_2 v_2 + \cdots +
  \alpha_d v_d + \alpha_{d+1} v_{d+1} + \cdots + \alpha_r v_r),
  \]\[ 
  X_2 = A_2 + \frac{\Z}{e_2}(\beta_1 w_1 + \beta_2 w_2 + \cdots +
  \beta_d w_d + \beta_{r+1} v_{r+1} + \cdots + \beta_s v_s)
  \] 
  where $\alpha_i, \beta_j \in \Z$, $\alpha_i$ divides $e_1$,
  $\beta_j$ divides $e_2$, $A_i = \R(X_i)$ and $e_1, e_2$ are
  relatively prime integers. For $i = 1, \ldots, d$, let $f_i =
  \gcd(e_2 \alpha_i, e_1 \beta_i)$. Then there exists a decomposition
  basis $\{v_1', w_1', \ldots, v_d', w_d', v_{d+1}, \ldots,  v_r,
  v_{r+1}, \ldots v_s\}$ of $A_1 \oplus A_2 = \R(X_1 \oplus X_2)$ such
  that with
  $$A = \tau_1 v_1' \oplus \cdots \oplus \tau_d v_d' \oplus \tau_{d+1}
  v_{d+1} \oplus \cdots \oplus \tau_s v_s,$$
  we have 
\begin{eqnarray*}
      X_1\oplus X_2 &=& Y \oplus \tau_1 w_1' \oplus \cdots \oplus
      \tau_d w_d' \quad\text{where}\\
      Y &=& A + \frac{\Z}{e_1 e_2} 
  \left(\sum_{i=1}^d f_i v_i' + \sum_{i=d+1}^r e_2 \alpha_i v_i
  + \sum_{i=r+1}^s e_1 \beta_i v_i \right).
\end{eqnarray*}
Furthermore, if $X_1$ and $X_2$ are both indecomposable, then so is
$Y$. All generator coefficients of $Y$ are divisors of $e_1 e_2$. In
passing from the decomposition $X_1 \oplus X_2$ to $Y \oplus \tau_1
w_1' \oplus \cdots \oplus \tau_d w_d'$ the number of summands does not
decrease. 
\end{lem} 

In \lemref{sums of rigid crq groups} we assume that the coefficients
$\alpha_i$ are factors of $e_1$ and the coefficients $\beta_i$ are
factors of $e_2$. By \lemref{integral coordinates} this can be done
without loss of generality, also, in examples, this will always be the
case.

\begin{proof} By \corref{direct sum of two groups} 
  \[\textstyle
  X_1 \oplus X_2 = (A_1 \oplus A_2) + \Z \frac{1}{e_1 e_2}
  \left(\sum_{i=1}^d (e_2 v_i + e_1 w_i) + \sum_{i=d+1}^r e_2 \alpha_i v_i
  + \sum_{i=r+1}^s e_1 \beta_i v_i \right)
  \]  
  and $A_1 \oplus A_2 = \R(X_1 \oplus X_2)$. 

  For $i = 1, \ldots, d$, we defined $f_i = \gcd(e_2 \alpha_i, e_1
  \beta_i)$. Then $f_i = \alpha_i \beta_i$ and setting $e_2 \alpha_i =
  f_i \alpha_i'$, $e_1 \beta_i = f_i \beta_i'$, there exist integers
  $a_i,b_i$ such that $a_i \alpha_i' + b_i \beta_i' = 1$ and hence we
  have that $\tau_i v_i \oplus \tau_i w_i = \tau_i (\alpha_i' v_i +
  \beta_i' w_i) \oplus \tau_i (-b_i v_i + a_i w_i)$.  Setting $v_i' =
  \alpha_i' v_i + \beta_i' w_i$ and $w_i' = -b_i v_i + a_i w_i$, we
  have the new decomposition basis $\{v_1', w_1', \ldots, v_d', w_d',
  v_{d+1}, \ldots, v_r, v_{r+1}, \ldots v_s\}$ of $A_1 \oplus A_2 =
  \R(X_1 \oplus X_2)$. Let $Y$ be as stated. Then $X_1 \oplus X_2 = Y
  \oplus \tau_1 w_1' \oplus \cdots \oplus \tau_d w_d'$ and by
  \thmref{product formula} $A = \R(Y)$.

  For clarity note that $\Tc(Y) = \Tc(X_1) \cup \Tc(X_2)$ and
  $\Tc(X_1) \cap \Tc(X_2) = \{\tau_1, \ldots, \tau_d\} \neq \emptyset$
  which makes it possible to merge the generators of $X_1$ and $X_2$.

  Suppose now that $X_1$ and $X_2$ are indecomposable. By
  \propref{indecomposable via frame} the frames of $X_1$ and $X_2$ are
  connected. We will show that the frame of $Y$ is connected. To do so
  recall that $\mu_\tau(Y) = \mu_\tau(X_1) \mu_\tau(X_2)$. So for any
  $\tau \in \Tc(X_1) $, we have that $\mu_\tau(X_1)$ and 
  $\mu_\tau(X_2)$ divide $\mu_\tau(Y)$. This shows that any edge
  between types of either the frame of $X_1$ or $X_2$ is also an edge
  in the frame of $Y$. As the frames of $X_1$ and $X_2$ are connected
  and $\Tc(X_1) \cap \Tc(X_2) \ne \emptyset$ it follows that the frame
  of $Y$ is connected and $Y$ is indecomposable.
\end{proof}

We use our previous shorthand to display the content of
\lemref{sums of rigid crq groups}. The formula 
  \[\textstyle
  [\tau_1\, \ldots \tau_d\, \tau_{d+1}\, \ldots\, \tau_r] \oplus 
  [\tau_1\, \ldots \tau_d\, \tau_{r+1}\, \ldots  \tau_s] =
  [\tau_1\, \ldots \tau_d\, \tau_{d+1}\, \ldots\, \tau_s] \oplus 
  [\tau_1] \oplus \cdots \oplus [\tau_d]
  \] 
correctly reflects the content of \lemref{sums of rigid crq
  groups}. We frequently use this formula without explicit reference. 

\exref{from (2,2) to (2,1,1)} explicitly describes $[\tau_1 \,
\tau_2] \oplus [\tau_1 \, \tau_2] = [\tau_1 \, \tau_2] \oplus
[\tau_1] \oplus [\tau_2]$.

\begin{ex}\label{from (2,2) to (2,1,1)} \rm There are groups realizing
  $\calS(4,2) = \{(2,2), (2,1,1)\} $. 

  In fact, let $p_1, p_2$ be relatively prime integers $\geq 2$, let
  $\tau_1, \tau_2$ be incomparable $p_1p_2$--free types, and let
  \[
  X_1 = (\tau_1 v_1 \oplus \tau_2 v_2) + \Z \frac{1}{p_1} (v_1 + v_2)
  \text{ and }
  X_2 = (\tau_1 w_1 \oplus \tau_2 w_2) + \Z \frac{1}{p_2} (w_1 + w_2).
  \]  
  Then $\R(X_1) = \tau_1 v_1 \oplus \tau_2 v_2$ and $\R(X_2) = \tau_1
  w_1 \oplus \tau_2 w_2$ and $X_1$ and $X_2$ are indecomposable. 
  With $u_1 p_1 + u_2 p_2 = 1$,
  \[\textstyle
  \begin{array}{cc} 
    X = X_1 \oplus X_2 = \left(\tau_1 (-u_1 v_1 + u_2 w_1) \oplus \tau_2
      (- u_1 v_2 + u_2 w_2)\right) \oplus \\
    \oplus
    \left(\left(\tau_1 (p_2 v_1 + p_1 w_1) \oplus 
                \tau_2 (p_2 v_2 + p_1 w_2)\right)
      + \Z \frac{1}{p_1 p_2} \left((p_2 v_1 + p_1 w1) + (p_2 v_2 +
        p_1 w_2)\right)\right),
  \end{array}
  \]
  so $X$ realizes $(1,1,2)$ in addition to $(2,2)$. 
\end{ex} 

\begin{proof} The summands $X_1$ and $X_2$ are indecomposable by
  \propref{indecomposable via frame} and the completely decomposable
  base groups are the regulators in each case (\thmref{product
    formula}. Now $\frac{1}{p_1} (v_1 + v_2) + \frac{1}{p_2} (w_1 +
  w_2) = \frac{1}{p_1 p_2} \left((p_2 v_1 + p_1 w_1) + (p_2 v_2 + p_1
    w_2)\right)$. Set $v_1' = p_2 v_1 + p_1 w_1$, $v_2' = p_2 v_2 +
  p_1 w_2$, $w_1' = -u_1 v_1 + u_2 w_1$, and $w_2' = -u_1 v_2 + u_2
  w_2$. Then (change of basis) $\tau_1 v_1 \oplus \tau_1 w_1 = \tau_1
  v_1' \oplus \tau_1 w_1'$ and $\tau_2 v_2 \oplus \tau_2 w_2 = \tau_2
  v_2' \oplus \tau_2 w_2'$. Hence $X = \left(\tau_1 w_1' \oplus \tau_2
    w_2'\right) \oplus \left(\left(\tau_1 v_1' \oplus \tau_2
      v_2'\right) + \Z \frac{1}{p_1 p_2} (v_1' + v_2')\right)$.
\end{proof} 

A Corner group is of the form $X = [\tau_1 \, \cdots \, \tau_\ell]
\oplus [\tau_1] \oplus \cdots \oplus [\tau_1]$. \propref{c(2n,n)
  realized} shows that certain Corner families of partitions may be
realized by groups that are not Corner groups.

\begin{prop}\label{c(2n,n) realized} The group $X_{2n} := [\tau_1 \
  \tau_2 ] \oplus [\tau_2 \ \tau_3 ] \oplus \cdots \oplus [\tau_n \
  \tau_{n+1}]$ realizes $\calC(2n,n)$.
\end{prop}

\begin{proof} Let $(r_1, r_2, \ldots, r_n) \in \calC(2n,n)$ where
  $\forall\, i \leq m : r_i \geq 2$ and $\forall\, i > m : r_i=1$.
  Note that $r_1 + \cdots + r_m + (n-m) = 2n$, so $r_1 + \cdots + r_m
  = n+m$ and $(r_1-1) + \cdots + (r_m-1) = n$. The group $X_{2n}$ has
  $n$ summands of rank $2$. We partition the $n$ summands of $X_{2n}$
  into groups of $r_1-1, r_2-1, \ldots, r_m-1$ summands. A typical
  summand is of the form $[\sigma_1 \ \sigma_2] \oplus [\sigma_2 \
  \sigma_3] \oplus \cdots \oplus [\sigma_{r-1} \ \sigma_r] = [\sigma_1 \
  \sigma_2 \cdots \sigma_r] \oplus [\sigma_2] \oplus \cdots \oplus
  [\sigma_{r-1}]$ where $[\sigma_1 \ \sigma_2 \cdots \sigma_r]$ is a
  summand of rank $r$. Note that in combining these summands the total
  number of summands has not changed. Doing this for each summand of
  the partition we get summands of rank $r_1, r_2, \ldots, r_m$ and
  $n-m$ summands of rank $1$, a total of $n$ summands. The partition
  $(r_1, r_2, \ldots, r_n) \in \calC(2n,n)$ has been realized by our
  particular decomposition of $X_{2n}$.
\end{proof}

\section{Realizing families $S(n,k)$}

The following proposition illustrates why some groups cannot realize
$S(n,k)$. 

\begin{prop}\label{no realization} Suppose that $G = Y \oplus Z$ is a
  decomposition with fully invariant summands neither of which is
  completely decomposable. Let $n_G = \rk(G)$, $n_Y = \rk(Y)$, and
  $n_Z = \rk(Z)$.  If $Y$ realizes $\calS(n_Y,k_Y)$, and $Z$ realizes
  $\calS(n_Z,k_Z)$.  then $\calP(G) = \calS(n_Y,k_Y) \times
  \calS(n_Z,k_Z) \subsetneqq \calS(n_G,k_Z)$, so $G$ does not realize
  $\calS(n_G,k_Z)$.
\end{prop}

\begin{proof} By \propref{unique hook}(2) we have that
  $\calP(Y) = \calS(n_Y,k_Y)$ and $\calP(Z) = \calS(n_Z,k_Z)$ and by
  \propref{direct sum of fully invariant subgroups}(2) $\calP(X) =
  \calP(Y) \times \calP(Z) = \calS(n_Y,k_Y) \times \calS(n_Z,k_Z)$.
  Without loss of generality, suppose $k_Z\geq k_Y$. By way of
  contradiction assume that $\calS(n_Y,k_Y) \times \calS(n_Z,k_Z) =
  \calS(n_X,k_Z)$. Then $X$ realizes $\calS(n_X,k_Z)$ and, in
  particular, $(k_Z,1^{n_X-k_Z}) \in \calS(n_X,k_Z)$. It follows that
  $(1^{n_Y}) \in \calP(Y)$ and $Y$ is completely decomposable, or else
  $(1^{n_Z}) \in \calP(Z)$ and $Z$ is completely decomposable, either
  way a contradiction.
\end{proof}

Our main result \thmref{complete answer} determines completely which
families $\calS(n,k)$ can be realized in $\calB$ and which cannot.

\begin{thm}\label{complete answer} \rm \hfill 
\begin{enumerate} 
\item The families $\calS(n,n), \calS(n,1), \calS(n,2), \calS(n,n-1)$
  are realizable in $\calB$.
\item The special families $\calS(5,3), \calS(6,4), \calS(4,2)$ can be
  realized in $\calB$.
\item All other hooked families $\calS(n,k)$ are not realizable in
  $\calB$. Specifically, the families
  $\calS(n,k)$ for $k \geq 3, n \geq 2k$, 
  $\calS(n,6)$ for $n \geq 10$, 
  $\calS(7,4)$,
  $\calS(7,5)$, $\calS(8,5)$, $\calS(9,5)$, 
  $\calS(8,6)$, $\calS(9,6)$, 
  $\calS(9,7)$ cannot be realized in $\calB$. 
\end{enumerate} 
\end{thm}

\begin{proof} Details of the proof are contained in the following
  lemmas, corollaries and examples:

(1) \lemref{special cases}.

(2) \exref{S53}, \exref{S64}, \exref{from (2,2) to (2,1,1)}.

(3) \corref{fail for n >= 2m m >= 3}, \exref{S106fail},
\corref{S(n,6)fail}, \exref{S74fail}, \exref{S75fail},
\exref{S86fail}, \exref{S97fail}.
\end{proof} 

\begin{lem}\label{special cases} \hfill 
\begin{enumerate}
\item $S(n,1) = \{(1, \ldots, 1)\}$ is realized by any completely
  decomposable group of rank $n$.
\item $S(n,n) = \{(n)\}$ is realized by any indecomposable
  $\calB$--group of rank $n$.
\item The set of partitions $\calS(n,2)$ can be realized.
\item $\calS(n,n-1)$ can be realized by suitable Corner groups.
\end{enumerate} 
\end{lem} 

\begin{proof} (3) First note that $P \in S(n,2)$ means that $P =
  (2,\ldots,2,1,\ldots,1) = (2^{i},1^{n-2i})$. So, for $n$ even,
  $S(n,2) = \{(2^{n/2}), \ldots, (2,1^{n-2})\}$. Using
  \lemref{sums of rigid crq groups} repeatedly we have for $n$ even,
  $n[\tau_1,\tau_2] = [\tau_1] \oplus [\tau_2] \oplus
  (n-1)[\tau_1,\tau_2] = \cdots = [\tau_1,\tau_2] \oplus (n/2 -
  1)[\tau_1] + (n/2 - 1)[\tau_2]$. For odd $n$ we have an obvious
  variant.

  (4) $\calS(n,n-1) = \{(n-1,1), (n-2,2), \ldots\}$ can
  be realized by Corner's Theorem.
\end{proof}

\begin{rem} \rm In all examples that follow we assume that the groups
  are given in the form of \lemref{integral coordinates} so that we
  have a handy formula for the invariants. The naming of the elements
  of the decomposition bases are irrelevant as only the associated
  types matter, i.e., in two decompositions of a group the $v_1$ in
  one decomposition and the $v_1$ in the other decomposition need not
  be the same. 
\end{rem} 

\begin{ex}\rm\label{S53}\rm Let $p_1, p_2, p_3$ be pairwise relatively
  prime integers $\geq 2$, let $\tau_1, \tau_2, \tau_3$ be pairwise
  incomparable $p_1 p_2 p_3$--free types, and let
\begin{enumerate} 
\item $R = \tau_1 v_1 \oplus \tau_2 v_2 \oplus \tau_3 v_3
  \oplus \tau_1 w_1 \oplus \tau_2 w_2$,
\item $X_1 = (\tau_1 v_1 \oplus \tau_2 v_2) + \Z \frac{1}{p_2
    p_3}(v_1 + v_2)$,
\item $X_2 = (\tau_1 w_1 \oplus \tau_3 v_3) + \Z
  \frac{1}{p_1}(w_1 + v_3)$, and 
\item $X_3 = \tau_2 w_2$. 
\end{enumerate}
Then $X = X_1 \oplus X_2 \oplus X_3$ realizes $S(5,3) =
\{(3,1,1),(3,2),(2,2,1)\}$. 

\begin{proof} By \lemref{sums of rigid crq groups}
  $X_1 \oplus X_2 = [\tau_1 \ \tau_2] \oplus [\tau_1 \ \tau_3] \oplus
  [\tau_2] = [\tau_1 \ \tau_2 \ \tau_3] \oplus [\tau_1] \oplus [\tau_2]$
  Thus $X = X_1 \oplus X_2 \oplus X_3$ realizes $(2,2,1)$ and
  $(3,1,1)$. 

  Let 
\begin{enumerate} 
\item $Y_1 = (\tau_1 v_1 \oplus \tau_2 v_2 \oplus \tau_3 v_3) + \Z
  \frac{1}{p_1 p_2}(v_1 + p_1 v_2 + p_2 v_3)$,
\item $Y_2 = (\tau_1 w_1 \oplus \tau_2 w_2) + \Z \frac{1}{p_3}(w_1 +
  w_2)$. 
\end{enumerate} 
  Then 

  $Y_1 \oplus Y_2 = [(\tau_1 v_1'' \oplus \tau_2 v_2'' \oplus \tau_3 v_3) +
  \frac{\Z}{p_1 p_2 p_3}(v_1'' + p_1 v_2'' + p_2 p_3 v_3)] \oplus
  \tau_1 w_1'' \oplus \tau_2 w_2''$ where $\tau_1 v_1 \oplus \tau_1 w_1
  = \tau_1 v_1'' \oplus \tau_1 w_1''$, and $\tau_2 v_2 \oplus \tau_2 w_2
  = \tau_2 v_2'' \oplus \tau_2 w_2''$.

  Then $Y = Y_1 \oplus Y_2$ realizes $(3,2)$ by definition.

  So far $X = X_1 \oplus X_2 \oplus X_3$ and $Y = Y_1 \oplus Y_2$ are
  independently constructed groups although $R = \R(X) = \R(Y)$ and
  they are both contained in $\Q R$, the divisible hull of $R$. We will
  show that they are nearly isomorphic and therefore they share their
  decomposition properties.  

  Using the Product Formula (\thmref{product formula}) we immediately
  see that 
  \[\textstyle
  \mu_{\tau_1}(X) = p_1 p_2 p_3,\quad \mu_{\tau_2}(X) = p_2 p_3,\quad
  \mu_{\tau_3}(X) = p_1.
  \] 
  and 
  \[\textstyle
  \mu_{\tau_1}(Y) = p_1 p_2 p_3,\quad \mu_{\tau_2}(Y) = p_2 p_3,\quad
  \mu_{\tau_3}(Y) = p_1.  
  \] 
  Hence $X \nriso Y$, and we are done.
\end{proof} 
\end{ex} 

\begin{ex}\label{S64}\rm The $\calB$--group $X = X_1 \oplus X_2
  \oplus X_3$ where  
\begin{enumerate} 
\item $X_1 = (\tau_1 v_1 \oplus \tau_2 v_2) + \Z \frac{1}{5 \cdot
  11}(v_1 + v_2)$, \newline
  where $\R(X_1) = \tau_1 v_1 \oplus \tau_2
  v_2$ and $\{v_1, v_2\}$ is a $5 \cdot 11$--basis of $\R(X_1)$,
\item $X_2 = (\tau_1 w_1 \oplus \tau_3 v_3) + \Z \frac{1}{7}(w_1 +
  v_3)$, \newline
  where $\R(X_2) = \tau_1 w_1 \oplus \tau_3 v_3$ and $\{w_1,
  v_3\}$ is a $7$--basis of $\R(X_2)$,
\item $X_3 = (\tau_2 w_2 \oplus \tau_4 v_4) + \Z \frac{1}{3}(w_2 +
  v_4)$, \newline
  where $\R(X_3) = \tau_2 w_2 \oplus \tau_4 v_4$ and $\{w_2,
  v_4\}$ is a $3$--basis of $\R(X_3)$.
\end{enumerate} 
  realizes $\calS(6,4)$.
\end{ex} 

\begin{proof} For easy reference note that $$\calS(6,4) = \{(4,2),
  (4,1,1), (3,3), (3,2,1), (2,2,2)\}.$$ It is assumed that all
  critical types $\tau_i$ are $3 \cdot 5 \cdot 7 \cdot 11$--free.  By
  definition $X := [\tau_1 \, \tau_2] \oplus [\tau_1 \, \tau_3] \oplus
  [\tau_2 \, \tau_4] = [\tau_1 \, \tau_2 \, \tau_3] \oplus [\tau_1]
  \oplus [\tau_2 \, \tau_4] = [\tau_1 \, \tau_2 \, \tau_3 \, \tau_4]
  \oplus [\tau_1] \oplus [\tau_2]$ showing that any such $X$ will
  indeed realize $(4,1,1), (3,2,1)$ and $(2,2,2)$.  We record the
  invariants of $X$.
\begin{equation*}\label{222 invariants}
  \mu_{\tau_1}(X) = 5 \cdot 7 \cdot 11, \quad \mu_{\tau_2}(X) = 5
  \cdot 11 \cdot 3, \quad \mu_{\tau_3}(X) = 7, 
  \quad \mu_{\tau_4}(X) = 3. 
\end{equation*}

Next let $Y = Y_1 \oplus Y_2$ where 
\begin{enumerate} 
\item $Y_1 = (\tau_1 v_1 \oplus \tau_2 v_2 \oplus \tau_3 v_3 \oplus
  \tau_4 v_4) + \Z \frac{1}{7 \cdot 3}(3 v_1 + 7 v_2 +
  3 v_3 + 7 v_4)$,
\item $Y_2 = (\tau_1 w_1 \oplus \tau_2 w_2) + \Z
  \frac{1}{5 \cdot 11}(w_1 + w_2)$. 
\end{enumerate}
It is easily checked that $\R(Y_1) = \tau_1 v_1 \oplus \tau_2 v_2
\oplus \tau_3 v_3 \oplus \tau_4 v_4$, $\R(Y_2) = \tau_1 w_1 \oplus
\tau_2 w_2$ and $Y_1, Y_2$ are indecomposable. The invariants of $Y$
are as follows.
\begin{eqnarray*}
  \mu_{\tau_1}(Y) &=& \frac{7 \cdot 3}{3} \cdot 5 \cdot 11 =  5 \cdot 11
  \cdot 7, \quad \mu_{\tau_2}(Y) = \frac{7 \cdot 3}{7} \cdot 5 \cdot 11 = 5
  \cdot 11 \cdot 3, \\
  \mu_{\tau_3}(Y) &=& \frac{7 \cdot 3}{3} = 7, \quad
  \mu_{\tau_4}(Y) = \frac{7 \cdot 3}{7} = 3.
\end{eqnarray*}

We see that $X$ and $Y$ have the same regulator and the same
invariants, hence $X \nriso Y$ and $X$ realizes $(4,2)$. 

Finally, let $Z = Z_1 \oplus Z_2$ where 
\begin{enumerate} 
\item $Z_1 = (\tau_1 v_1 \oplus \tau_2 v_2 \oplus \tau_3 v_3) +
  \Z \frac{1}{7 \cdot 5} (v_1 + 7 v_2 + 5 v_3)$,
\item $Z_2 = (\tau_1 w_1 \oplus \tau_2 w_2 \oplus \tau_4 v_4) +
  \Z \frac{1}{3 \cdot 11} (3 w_1 + w_2 + 11 v_4)$.
\end{enumerate}
It is routine to verify that $\R(Z_1) = \tau_1 v_1 \oplus
\tau_2 v_2 \oplus \tau_3 v_3$ and $\R(Z_2) = \tau_1 w_1 \oplus \tau_2
w_2 \oplus \tau_4 v_4$, hence $\R(Z) = \R(X)$. 
We now compute invariants. 
\begin{eqnarray*}
  \mu_{\tau_1}(Z) &=& 7 \cdot 5 \cdot \frac{3 \cdot 11}{3} = 7 \cdot 5
  \cdot 11, \quad
  \mu_{\tau_2}(Z) = \frac{7 \cdot 5}{7} \cdot 3 \cdot 11 = 5 \cdot 3
  \cdot 11, \\ 
  \mu_{\tau_3}(Z) &=& \frac{5 \cdot 7}{5} = 7, \quad
  \mu_{\tau_4}(Z) = \frac{3 \cdot 11}{11} = 3.
\end{eqnarray*}
Hence $Z_1, Z_2$ are indecomposable, $Z \nriso X$ and $X$ realizes
$(3,3)$.
\end{proof}

It is intuitively clear that ``larger'' sets of partitions cannot be
realized if ``smaller'' ones cannot.

\begin{prop}\label{getting worse}
  Let $n' > n \geq k$. If $\calS(n',k)$ can be realized, then so can
  $\calS(n,k)$. Equivalently, if $\calS(n,k)$ cannot be realized, then
  neither can $\calS(n',k)$.
\end{prop}

\begin{proof} Let $(r_1, \ldots, r_t) \in \calS(n,k)$ and $n' = n+1$.
  Then $k + t \leq n+1$. Consider $P = (r_1, \ldots, r_t, 1) \in
  \calP(n')$. Then $k + (t+1) \leq n+1 + 1 = n' + 1$, so $P \in
  \calS(n',k)$. Suppose $\calS(n',k)$ can be realized.  Then there is
  a group $X$ and a decomposition $X = Y_1 \oplus \cdots \oplus Y_t
  \oplus Y$ such that $\rk Y_i = r_i$, the $Y_i$ are indecomposable
  and $\rk(Y) = 1$. Then $X/Y$ realizes $(r_1, \ldots, r_t)$. Repeat
  for arbitrary $n' > n$ until $\calS(n,k)$ is reached. 
\end{proof} 

\begin{prop}\label{not realizable bunch} There is no $\calB$--group
  realizing $\calS(2m,m)$ provided $m \geq 3$.
\end{prop}

\begin{proof} Suppose $X \in \calB$ realizes $\calS(2m,m)$ for $m \geq
  3$. Note that $\calS(2m,m)$ contains the partitions $(m,1^m)$ and
  $(m,m)$.  Then $X$ cannot have completely decomposable fully
  invariant summands (\lemref{completely decomposable tau homogeneous
    fully invariant summand}) and cannot be the direct sum of fully
  invariant subgroups (\lemref{no realization}). \lemref{pinning down
    critical types} implies that $\Tc(X) = \{\tau_1, \ldots, \tau_m\}$,
  $X = [\tau_1 \, \cdots \, \tau_m] \oplus [\tau_1] \oplus \cdots
  [\tau_m]$, and by \lemref{bound for ranks} $\forall\, \tau \in
  \Tc(X): \rk(X(\tau)) \leq 2$. The group $X$ also realizes $(2^m)$.
  We claim that, without loss of generality,
\begin{equation}\label{realization 2 hoch m}
  X = [\tau_1 \ \tau_2; p_1] \oplus [\tau_2 \ \tau_3; p_2] \oplus
  \cdots \oplus [\tau_{m-1} \ \tau_m; p_{m-1}] 
  \oplus [\tau_m \ \tau_1; p_m]
\end{equation}
  where

\begin{eqnarray}\label{}
  [\tau_1 \ \tau_2] &=& (\tau_1 w_1 \oplus \tau_2 v_2) +
  \frac{\Z}{p_1}(w_1 + v_2), \nonumber \\
  &\vdots& \nonumber \\
  {[\tau_i \ \tau_{i+1}]} &=& (\tau_i w_i \oplus \tau_{i+1} v_{i+1}) +
  \frac{\Z}{p_i}(w_i + v_{i+1}), \label{rank 2 summands} \\
  &\vdots& \nonumber \\
  {[\tau_m \ \tau_1]} &=& (\tau_m w_m \oplus \tau_1 v_1) +
  \frac{\Z}{p_m}(w_m + v_1). \nonumber
\end{eqnarray}

In fact, the first rank--$2$ summand of $X$ may be assumed to be
$[\tau_1\ \tau_2; p_1]$. The second rank--$1$ summand of $\R(X)$ of
type $\tau_2$ appears in a summand $[\tau_2\ \tau; p_2]$. The type
$\tau$ cannot be $\tau_1$ because this would create a decomposition
with fully invariant summands. So without loss of generality $\tau =
\tau_3$.  Continuing in this fashion we obtain the claimed
decomposition \eqref{realization 2 hoch m}. The summands may be
written in the form \eqref{rank 2 summands} by \lemref{integral
  coordinates}. Thus we have the following invariants $\mu_i =
\mu_i(X)$.
\begin{equation*}
  \mu_1 = p_1 p_m,\ \mu_2 = p_2 p_1,\ \ldots,\ \mu_{m-1} = p_{m-1}
  p_{m-2}, \ \mu_m = p_m p_{m-1}.
\end{equation*}
The group $X$ also realizes $(2^{m-1},1,1)$ and hence may be assumed
to be of the form 
\begin{equation}\label{realization of 2 hoch m-1,1,1}
  X = [\tau_i] \oplus [\tau_j] \oplus [\tau_i\ \sigma_1;e_1] \oplus
  [\sigma_1\ \sigma_2; e_2] \oplus \cdots \oplus [\sigma_{m-3}\
  \sigma_{m-2}; e_{m-2}] \oplus [\sigma_{m-1}\ \tau_j;e_{m-1}]
\end{equation}
where without loss of generality $i < j$. Again assume the
presentation of \lemref{integral coordinates} and obtain the invariants 
\begin{equation*}
  \mu_i = e_1,\quad \mu_j = e_{m-1},\quad \mu_{\sigma_s}(X) = e_s e_{s+1}
  \text{ otherwise.} 
\end{equation*}
We obtain in particular that 
  \[\textstyle
  \mu_i = e_1 = p_i p_{i-1}, \quad \mu_j = e_{m-1} = p_j p_{j-1}.
  \]
  Suppose that $\sigma_1 = \tau_k$, so $k \ne i$. Assume first that $k
  \ge 2$.  Then $\mu_k = p_k p_{k-1} = e_1 e_2 = p_i p_{i-1} e_2$. So
  $p_i$ divides $p_k p_{k-1}$ and it follows that $p_i = p_{k-1}$, so
  $i=k-1$. Hence $p_k p_{k-1} = p_{k-1} p_{k-2} e_2$, and $p_{k-2}
  \mid p_k$, a contradiction. This leaves the case $k=1$. In this case
  $\mu_1 = p_1 p_{m} = e_1 e_2 = p_i p_{i-1} e_2$. So $p_i$ divides
  $p_1 p_{m}$ and it follows that $p_i = p_m$ and $p_{i-1} = p_1$, so
  $i=2$. Hence $p_1 p_{m} = p_2 p_{1} e_2$, and $p_{2} \mid p_m$,
  again a contradiction as $m \geq 3$.
\end{proof} 

\begin{cor}\label{fail for n >= 2m m >= 3} There is no $\calB$--group
  realizing $\calS(n,m)$ for $m \geq 3$ and $n \geq 2m$.
\end{cor}

\begin{proof} \propref{not realizable bunch} and \propref{getting
    worse}.
\end{proof} 

\begin{ex}\label{S106fail}\rm The hooked family $\calS(10,6)$ cannot be
  realized by $\calB$--groups.
\end{ex}

\begin{proof} For easy reference
{\small
  \[\textstyle
  \calS(10,6) = 
  \] 
  \[
  \{(2,2,2,2,2), (3,2,2,2,1), (3,3,2,1,1), (4,2,2,1,1), (4,3,1,1,1),
  (5,2,1,1,1), (6,1,1,1,1),
  \] 
  \[\textstyle
  (3,3,3,1), (3,3,2,2), (4,2,2,2), (4,3,2,1), (4,4,1,1), (5,2,2,1),
  (5,3,1,1), (6,2,1,1),
  \] 
  \[\textstyle
  (4,4,2), (4,3,3), (5,4,1), (5,3,2), (6,2,2), (6,3,1), (5,5), (6,4)\}
  \] }

  Suppose that $X$ realizes $(6,1,1,1,1)$ and $(6,4)$.  Then $\Tc(X) =
  \{\tau_1, \tau_2, \tau_3, \tau_4, \tau_5, \tau_6\}$ and $X =
  [\tau_1\, \tau_2\, \tau_3\, \tau_4\, \tau_5\, \tau_6] \oplus
  [\tau_1] \oplus [\tau_2] \oplus [\tau_3] \oplus [\tau_4]$ by
  \lemref{pinning down critical types}. Note that the types $\tau_1,
  \ldots, \tau_4$ are interchangeable, and $\tau_5, \tau_6$ are
  interchangeable. Assume further that $X$ realizes $(2,2,2,2,2)$. We
  claim that after renaming types if necessary we may assume that $X =
  [\tau_1 \, \tau_2] \oplus [\tau_2 \, \tau_3] \oplus [\tau_3 \,
  \tau_4] \oplus [\tau_4 \, \tau_5] \oplus [\tau_1 \, \tau_6]$. In
  fact, $\tau_5$ and $\tau_6$ cannot belong to the same summand and
  must be paired with different types to avoid fully invariant
  summands, so without loss of generality $X = [\tau_4 \, \tau_5]
  \oplus [\tau_1 \, \tau_6] \oplus \cdots$. The remaining three
  summands must contain $\tau_1$ and $\tau_4$ once each and $\tau_2,
  \tau_3$ twice each. With proper labeling we get $X = [\tau_1 \,
  \tau_2] \oplus [\tau_2 \, \tau_3] \oplus [\tau_3 \, \tau_4] \oplus
  [\tau_4 \, \tau_5] \oplus [\tau_1 \, \tau_6]$. In order to compute
  invariants we need the explicit description of $X$. By
  \lemref{integral coordinates} we may assume that
\begin{eqnarray*}
  X &=& 
  \left[(\tau_1 v_1 \oplus \tau_2 v_2) + \frac{\Z}{p_1}(v_1 +
    v_2)\right] 
  \oplus \left[(\tau_2 w_2 \oplus \tau_3 v_3) + \frac{\Z}{p_2}(w_2 +
    v_3)\right] \\
  &\oplus& \left[(\tau_3 w_3 \oplus \tau_4 v_4) + \frac{\Z}{p_3}(w_3 +
    v_4)\right] 
  \oplus \left[(\tau_4 w_4 \oplus \tau_5 v_5) + \frac{\Z}{p_4}(w_4 +
    v_5)\right] \\
  &\oplus& \left[(\tau_1 w_1 \oplus \tau_6 v_6) + \frac{\Z}{p_5}(w_1 +
    v_6)\right],
\end{eqnarray*}
  where $\tau_1, \tau_2$ are $p_1$--free, $\tau_2, \tau_3$ are
  $p_2$--free, $\tau_3, \tau_4$ are $p_3$--free, $\tau_4, \tau_5$ are
  $p_4$--free, $\tau_1, \tau_6$ are $p_5$--free, and the $p_i$ are
  pairwise relatively prime integers.
  The invariants $\mu_i = \mu_i(X)$ are easily read off. 
  \[\textstyle
  \mu_1 = p_1 p_5, \quad \mu_2 = p_1 p_2, \quad \mu_3 = p_2 p_3, \quad
  \mu_4 = p_3 p_4, \quad \mu_5 = p_4, \quad \mu_6 = p_5.
  \]

  The partition $(4,4,1,1)$ cannot be realized by $X$. By way of
  contradiction suppose $X$ realizes $(4,4,1,1)$. The exceptional
  types $\tau_5$ and $\tau_6$ cannot be the types of the rank--$1$
  summands because they would be fully invariant. So $\tau_5, \tau_6$
  appear together in one of the rank--$4$ summands or they are
  distributed over the rank--$4$ summands. We have to consider two
  cases.

{\bf Case I.} $X = [\tau_1 \, \tau_2 \, \tau_3 \, \tau_4] \oplus
[\tau_1 \, \tau_2 \, \tau_5 \, \tau_6] \oplus [\tau_3] \oplus
[\tau_4]$

In this case $X = X_1 \oplus X_2 \oplus [\tau_3] \oplus [\tau_4]$
where 
  \[\textstyle
  X_1 = [\tau_1 v_1 \oplus \tau_2 v_2 \oplus \tau_3 v_3 \oplus \tau_4
  v_4] + \frac{\Z}{e_1} (\alpha_1 v_1 + \alpha_2 v_2 + \alpha_3 v_3 +
  \alpha_4 v_4) 
  \] 
  and 
  \[\textstyle
  X_2 = [\tau_1 w_1 \oplus \tau_2 w_2 \oplus \tau_5 v_5 \oplus \tau_6
  v_6] + \frac{\Z}{e_2} 
  (\beta_1 w_1 + \beta_2 w_2 + \beta_3 v_5 + \beta_4 v_6)
  \]
  We now see that $\mu_4 = p_3 p_4 = \frac{e_1}{\alpha_4}$ and
  $\mu_5 = p_4 = \frac{e_2}{\beta_3}$. This means that $p_4$ is a
  factor both of $e_1$ and $e_2$ contradicting $\gcd(e_1,e_2) = 1$. 

  {\bf Case II.} $X = [\tau_1 \, \tau_2 \, \tau_3 \, \tau_5] \oplus
  [\tau_1 \, \tau_2 \, \tau_4 \, \tau_6] \oplus [\tau_3] \oplus
  [\tau_4]$

In this case $X = X_1 \oplus X_2 \oplus [\tau_3] \oplus [\tau_4]$
where 
  \[\textstyle
  X_1 = 
  [\tau_1 v_1 \oplus \tau_2 v_2 \oplus \tau_3 v_3 \oplus \tau_5 v_5] 
  + \frac{\Z}{e_1}  
  (\alpha_1 v_1 + \alpha_2 v_2 + \alpha_3 v_3 + \alpha_5 v_5)
  \] 
  and 
  \[\textstyle
  X_2 = 
  [\tau_1 w_1 \oplus \tau_2 w_2 \oplus \tau_4 w_4 \oplus \tau_6 w_6] 
  + \frac{\Z}{e_2}  
  (\beta_1 w_1 + \beta_2 w_2 + \beta_4 w_4 + \beta_6 v_6)
  \]
  We now see that $\mu_3 = p_2 p_3 = \frac{e_1}{\alpha_3}$ and
  $\mu_4 = p_3 p_4 = \frac{e_1}{\beta_4}$. This means that
  $p_3$ is a factor both of $e_1$ and $e_2$, a contradiction.
\end{proof}

\begin{cor} \label{S(n,6)fail} $\calS(n,6)$ cannot be realized in
  $\calB$ for $n \geq 10$. 
\end{cor} 

\begin{ex}\label{S74fail}\rm The hooked family $\calS(7,4)$ cannot be
  realized in $\calB$. 
\end{ex} 

\begin{proof} For easy reference 
  \[\textstyle
  \calS(7,4) = \{(2,2,2,1),\, (3,2,1,1), \,
  (4,1,1,1), \, (4,2,1),\, (3,3,1),\, (3,2,2),\, (4,3)\}.
  \]
  Suppose that $X \in \calB$ realizes $\calS(7,4)$. Then $X = [\tau_1
  \, \tau_2 \, \tau_3 \, \tau_4] \oplus [\tau_1] \oplus [\tau_2]
  \oplus [\tau_3]$ (\lemref{pinning down critical types}). So there is
  the exceptional critical type $\tau_4$ that appears only once in
  decompositions. By assumption $X$ realizes $(2,2,2,1)$. The
  rank--$1$ summand cannot have type $\tau_4$ because, if so, it would
  be a fully invariant summand. Hence without loss of generality $X =
  [\tau_4 \, \tau_1] \oplus [\tau_1 \, \tau_2] \oplus [\tau_2 \,
  \tau_3] \oplus [\tau_3]$. Explicitly we have
\begin{eqnarray*}
  X &=& X_1 \oplus X_2 \oplus X_3 \oplus \tau_3 w_3 \quad \text{where} \\
  X_1 &=& [\tau_4 v_4 \oplus \tau_1 v_1] + \frac{\Z}{p_1}(v_4 + v_1), \\
  X_2 &=& [\tau_1 w_1 \oplus \tau_2 v_2] + \frac{\Z}{p_2}(w_1 + v_2), \\
  X_3 &=& [\tau_2 w_2 \oplus \tau_3 v_3] + \frac{\Z}{p_3}(w_2 + v_3),
\end{eqnarray*}
  and $\tau_1$ is $p_1 p_2$--free, $\tau_2$ is $p_2 p_3$--free,
  $\tau_3$ is $p_3$--free, and $\tau_4$ is $p_1$--free. We can now
  compute the invariants $\mu_i = \mu_{\tau_i}(X)$. 
  \[\textstyle
  \mu_1 = p_1 p_2, \quad \mu_2 = p_2 p_3, \quad \mu_3 = p_3, \quad
  \mu_4 = p_1.
  \]
  Our group $X$ also realizes $(3,2,2)$. There are two cases depending
  on whether the exceptional type $\tau_4$ belongs to the summand of
  rank $3$ or not.

  {\bf Case I.} $X = [\tau_4 \, \tau_1 \, \tau_2] \oplus [\tau_2
  \tau_3] \oplus [\tau_3 \, \tau_1]$ and, explicitly,
\begin{eqnarray*}
  X &=& X_1 \oplus X_2 \oplus X_3 \quad \text{where} \\
  X_1 &=& [\tau_4 v_4 \oplus \tau_1 v_1 \oplus \tau_2 v_2] 
  + \frac{\Z}{e_1}(\alpha_4 v_4 + \alpha_1 v_1 + \alpha_2 v_2), \\
  X_2 &=& [\tau_2 w_2 \oplus \tau_3 v_3] + \frac{\Z}{e_2}(w_2 + v_3), \\
  X_3 &=& [\tau_3 w_3 \oplus \tau_1 w_1] + \frac{\Z}{e_3}(w_3 + w_1),
\end{eqnarray*}
  where $\tau_2$ is $e_2$--free, $\tau_3$ is $e_2 e_3$--free, and
  $\tau_1$ is $e_3$--free. By way of invariants we obtain that 
  \[\textstyle
  \mu_2 = p_2 p_3 = \frac{e_1}{\alpha_2} e_2, \ 
  \mu_3 = p_3 = e_2 e_3.
  \]
  It follows that $\alpha_2 p_2 e_2 e_3 = \alpha_2 p_2 p_3 = e_1
  e_2$ and hence $e_3 \mid e_1$, a contradiction.
\smallskip

  {\bf Case II.} $X = [\tau_1 \, \tau_2 \, \tau_3] \oplus [\tau_3
  \tau_4] \oplus [\tau_1 \, \tau_2]$ and, explicitly,
\begin{eqnarray*}
  X &=& X_1 \oplus X_2 \oplus X_3 \quad \text{where} \\
  X_1 &=& [\tau_1 v_1 \oplus \tau_2 v_2 \oplus \tau_3 v_3] 
  + \frac{\Z}{e_1}(\alpha_1 v_1 + \alpha_2 v_2 + \alpha_3 v_3), \\
  X_2 &=& [\tau_3 w_3 \oplus \tau_4 v_4] + \frac{\Z}{e_2}(w_3 + v_4), \\
  X_3 &=& [\tau_1 w_1 \oplus \tau_2 w_2] + \frac{\Z}{e_3}(w_1 + w_2),
\end{eqnarray*}
where $\tau_1, \tau_4$ are $e_2$--free, and $\tau_1, \tau_2$ are
$e_3$--free. By way of invariants we obtain that
  \[\textstyle
  \mu_1 = p_1 p_2 = \frac{e_1}{\alpha_1} e_3, \ 
  \mu_4 = p_1 = e_2.
  \]
  Hence $e_2 p_2 = \frac{e_1}{\alpha_1} e_3$, an immediate
  contradiction as $e_1, e_2, e_3$ are pairwise relatively prime.
\end{proof} 

\begin{ex}\label{S75fail}\rm The hooked family $\calS(7,5)$ cannot be
  realized in $\calB$. Consequently, neither can $\calS(8,5)$ and
  $\calS(9,5)$.
\end{ex} 

\begin{proof} For easy reference 
  \[\textstyle
  \calS(7,5) = 
  \{(3,2,2),\, (3,3,1),\, (4,2,1),\, (5,1,1),\, (4,3),\, (5,2)
  \}.
  \]
  Suppose that $X \in \calB$ realizes $\calS(7,5)$. Then $X = [\tau_1
  \, \tau_2 \, \tau_3 \, \tau_4 \, \tau_5] \oplus [\tau_1] \oplus
  [\tau_2]$ (\lemref{pinning down critical types}). So there are the
  exceptional critical types $\tau_3, \tau_4, \tau_5$ that appear only
  once in decompositions. By assumption $X$ realizes $(3,2,2)$. Each
  rank--$2$ summand must contain exactly one of $\tau_1$ and $\tau_2$
  in order to avoid a fully invariant summand. Hence without loss of
  generality $X = [\tau_1 \, \tau_2 \, \tau_3] \oplus [\tau_1 \,
  \tau_4] \oplus [\tau_2 \, \tau_5]$. Explicitly we have
\begin{eqnarray*}
  X &=& X_1 \oplus X_2 \oplus X_3 \quad
  \text{where} \\ 
  X_1 &=& [\tau_1 v_1 \oplus \tau_2 v_2 \oplus \tau_3 v_3] +
  \frac{\Z}{p_1}(\alpha_1 v_1 + \alpha_2 v_2 + \alpha_3 v_3), \\
  X_2 &=& [\tau_1 w_1 \oplus \tau_4 v_4] + \frac{\Z}{p_2}(w_1 + v_4), \\
  X_3 &=& [\tau_2 w_2 \oplus \tau_5 v_5] + \frac{\Z}{p_3}(w_2 + v_5),
\end{eqnarray*}
where $\tau_1, \tau_4$ are $p_2$--free, and $\tau_2, \tau_5$ are
$p_3$--free. We can now compute the invariants $\mu_i =
\mu_{\tau_i}(X)$. 
  \[\textstyle
  \mu_1 = \frac{p_1}{\alpha_1} p_2, \quad 
  \mu_2 = \frac{p_1}{\alpha_2} p_3, \quad 
  \mu_3 = \frac{p_1}{\alpha_3}, \quad 
  \mu_4 = p_2, \quad 
  \mu_5 = p_3.
  \]
  It is easily seen that, in addition to $(3,2,2)$, our $X$ also
  realizes $(4,2,1)$ and $(5,1,1)$, but it does not realize $(3,3,1)$
  as we shall see, so not all of $\calC(7,3)$ is realized. Suppose
  that $X$ realizes $(3,3,1)$. So $X = Y_1 \oplus Y_2 \oplus Y_3$ with
  $Y_3 = \sigma w$ having rank $1$. Then $\sigma$ cannot be one of the
  exceptional type $\tau_3, \tau_4, \tau_5$ to avoid a fully invariant
  summand of rank $1$. So without loss of generality $X = [\tau_1 \,
  \tau_2 \, \tau_3] \oplus [\tau_2 \, \tau_4 \, \tau_5] \oplus \tau_1
  w_1$. Explicitly, we have
\begin{eqnarray*}
  X &=& Y_1 \oplus Y_2 \oplus Y_3 \quad\text{where} \\
  Y_1 &=& [\tau_1 v_1 \oplus \tau_2 v_2 \oplus \tau_3 v_3] 
  + \frac{\Z}{q_1}(\beta_1 v_1 + \beta_2 v_2 + \beta_3 v_3), \\
  Y_2 &=& [\tau_2 w_2 \oplus \tau_4 v_4 \oplus \tau_5 v_5] +
  \frac{\Z}{q_2}(\gamma_2 w_2 + \gamma_4 v_4 + \gamma_5 v_5).
\end{eqnarray*}
  By way of invariants we obtain that 
  \[\textstyle
  \mu_1 = \frac{p_1}{\alpha_1} p_2 = \frac{q_1}{\beta_1}, \ 
  \mu_4 = p_2 = \frac{q_2}{\gamma_4}, \  
  \mu_5 = p_3 = \frac{q_2}{\gamma_5}.
  \]
  From the last two statements show that $p_2 \mid q_2$ and $p_3 \mid
  q_2$, and hence $q_2 = p_2 p_3 p_1'$ for some $p_1'$. Using that
  $[X:R(X)] = p_1 p_2 p_3 = q_1 q_2$ it follows that $p_1 = q_1 p_1'$.
  From the equation for $\mu_1$, it now follows that $\beta_1 q_1 p_1' p_2
  = \alpha_1 q_1$, so $\beta_1 p_1' p_2 = \alpha_1$. We arrive at the
  contradiction $p_2 \mid \alpha_1 \mid p_1$.
\end{proof} 

\begin{ex}\label{S86fail}\rm The hooked families $\calS(8,6)$ and
  $\calS(9,6)$ cannot be realized in $\calB$.
\end{ex} 

\begin{proof} For easy reference 
  \[\textstyle
  \calS(8,6) = 
  \{(6,1,1),\, (5,2,1),\, (4,3,1),\, (4,2,2),\, (3,3,2),\,  (6,2),\,
  (5,3) ,\, (4,4)\}.
  \]
  Suppose that $X \in \calB$ realizes $\calS(8,6)$. Then $X = [\tau_1
  \, \tau_2 \, \tau_3 \, \tau_4 \, \tau_5 \, \tau_6] \oplus [\tau_1]
  \oplus [\tau_2]$. The critical types $\tau_3, \ldots, \tau_6$ 
  appear only once in decompositions.

  By assumption $X$ realizes $(4,2,2)$. It then also realizes
  $(6,1,1), (5,2,1), (4,3,1)$ Either one of $\tau_1, \tau_2$ appears
  in the rank-$4$ summand, or both do.  We have two cases:
\begin{enumerate} 
\item {\bf Case I.} $X = [\tau_1 \, \tau_3 \, \tau_4 \, \tau_5] \oplus
  [\tau_1 \, \tau_2] \oplus [\tau_2 \, \tau_6]$ and
\item {\bf Case II.} $X = [\tau_1 \, \tau_2 \, \tau_3 \, \tau_4] \oplus
  [\tau_1 \, \tau_5] \oplus [\tau_2 \, \tau_6]$.
\end{enumerate} 

{\bf Case I.} Explicitly we have
\begin{eqnarray*}
  X &=& X_1 \oplus X_2 \oplus X_3 \quad \text{where} \\ 
  X_1 &=& [\tau_1 v_1 \oplus \tau_3 v_3 \oplus \tau_4 v_4 \oplus
  \tau_5 v_5] + \frac{\Z}{p_1}(\alpha_1 v_1 + \alpha_3 v_3 + \alpha_4
  v_4 + \alpha_5 v_5), \\  
  X_2 &=& [\tau_1 w_1 \oplus \tau_2 v_2] + \frac{\Z}{p_2}(w_1 + v_2), \\
  X_3 &=& [\tau_2 w_2 \oplus \tau_6 v_6] + \frac{\Z}{p_3}(w_2 + v_6),
\end{eqnarray*}
  $\tau_1, \tau_2$ are $p_2$--free, and $\tau_2, \tau_6$ are $p_3$--free.
  We can now compute the invariants $\mu_i = \mu_{\tau_i}(X)$. 
  \[\textstyle
  \mu_1 = \frac{p_1}{\alpha_1} p_2, \quad 
  \mu_2 = p_2 p_3, \quad 
  \mu_3 = \frac{p_1}{\alpha_3}, \quad
  \mu_4 = \frac{p_1}{\alpha_4}, \quad 
  \mu_5 = \frac{p_1}{\alpha_5}, \quad
  \mu_6 = p_3.
  \]

{\bf Case II.} Explicitly we have
\begin{eqnarray*}
  X &=& X_1 \oplus X_2 \oplus X_3 \quad \text{where} \\ 
  X_1 &=& [\tau_1 v_1 \oplus \tau_2 v_2 \oplus \tau_3 v_3 \oplus
  \tau_4 v_4] + \frac{\Z}{p_1}(\alpha_1 v_1 + \alpha_2 v_2 + \alpha_3
  v_3 + \alpha_4 v_4), \\  
  X_2 &=& [\tau_1 w_1 \oplus \tau_5 v_5] + \frac{\Z}{p_2}(w_1 + v_5), \\
  X_3 &=& [\tau_2 w_2 \oplus \tau_6 v_6] + \frac{\Z}{p_3}(w_2 + v_6),
\end{eqnarray*}
  and $\tau_1, \tau_5$ are $p_2$--free, $\tau_2, \tau_6$ are $p_3$--free.
  We can now compute the invariants $\mu_i = \mu_{\tau_i}(X)$. 
  \[\textstyle
  \mu_1 = \frac{p_1}{\alpha_1} p_2, \quad 
  \mu_2 = \frac{p_1}{\alpha_2} p_3, \quad 
  \mu_3 = \frac{p_1}{\alpha_3}, \quad
  \mu_4 = \frac{p_1}{\alpha_4}, \quad 
  \mu_5 = p_2, \quad
  \mu_6 = p_3.
  \]

  Our group $X$ also realizes $(3,3,2)$. There are two cases depending
  on whether the types $\tau_1, \tau_2$ belongs to different summands
  of rank $3$ or not.

\begin{enumerate} 
\item {\bf Case A.} $X = [\tau_1 \, \tau_3 \, \tau_4] \oplus [\tau_2 \,
  \tau_5 \, \tau_6] \oplus [\tau_1 \, \tau_2]$
\item {\bf Case B.} $X = [\tau_1 \, \tau_2 \, \tau_3] \oplus [\tau_2 \, 
  \tau_4 \, \tau_5] \oplus [\tau_1 \, \tau_6]$ 
\end{enumerate} 

  {\bf Case A.} Explicitly, 
\begin{eqnarray*}
  X &=& X_1 \oplus X_2 \oplus X_3 \quad \text{where} \\
  X_1 &=& [\tau_1 v_1 \oplus \tau_3 v_3 \oplus \tau_4 v_4] 
  + \frac{\Z}{q_1}(\beta_1 v_1 + \beta_3 v_3 + \beta_4 v_4), \\
  X_2 &=& [\tau_2 v_2 \oplus \tau_5 v_5 \oplus \tau_6 v_6] +
  \frac{\Z}{q_2}(\gamma_2 v_2 + \gamma_5 v_5 + \gamma_6 v_6), \\ 
  X_3 &=& [\tau_1 w_1 \oplus \tau_2 w_2] + \frac{\Z}{q_3}(w_1 + w_2)
\end{eqnarray*}
  where $\tau_1, \tau_2$ are $q_3$--free. The invariants are
  \[\textstyle
  \mu_1 = \frac{q_1}{\beta_1} q_3, \quad
  \mu_2 = \frac{q_2}{\gamma_2} q_3, \quad
  \mu_3 = \frac{q_1}{\beta_3}, \quad 
  \mu_4 = \frac{q_1}{\beta_4}, \quad 
  \mu_5 = \frac{q_2}{\gamma_5}, \quad
  \mu_6 = \frac{q_2}{\gamma_6}.
  \]
\bigskip

  {\bf Case B.} Explicitly, 
\begin{eqnarray*}
  X &=& X_1 \oplus X_2 \oplus X_3 \quad \text{where} \\
  X_1 &=& [\tau_1 v_1 \oplus \tau_2 v_2 \oplus \tau_3 v_3 \oplus ] 
  + \frac{\Z}{q_1}(\beta_1 v_1 + \beta_2 v_2 + \beta_3 v_3), \\
  X_2 &=& [\tau_2 v_2 \oplus \tau_4 v_4 \oplus \tau_5 v_5] 
  + \frac{\Z}{q_2}(\gamma_2 w_2 + \gamma_4 v_4 + \gamma_5 v_5), \\
  X_3 &=& [\tau_1 w_1 \oplus \tau_6 v_6] + \frac{\Z}{q_3}(w_1 + v_6),
\end{eqnarray*}
where $\tau_6, \tau_1$ are $q_2$--free. The invariants are
  \[\textstyle
  \mu_1 = \frac{q_1}{\beta_1} q_3, \ 
  \mu_2 = \frac{q_1}{\beta_2} \frac{q_2}{\gamma_2}, \ 
  \mu_3 = \frac{q_1}{\beta_3}, \ 
  \mu_4 = \frac{q_2}{\gamma_4}, \
  \mu_5 = \frac{q_2}{\gamma_5}, \ 
  \mu_6 = q_3.
  \]

{\bf Case I.A.} The relevant invariants are
  \[\textstyle
  \mu_2 = p_2 p_3 = \frac{q_2}{\gamma_2} q_3, \quad \mu_5 =
  \frac{p_1}{\alpha_5} = \frac{q_2}{\gamma_5}, \quad
  \mu_6 = p_3 = \frac{q_2}{\gamma_6}.
  \] 

The equation for $\mu_6$ implies $q_2 = p_3 \gamma_6$. Together with
the equation for $\mu_5$ we get $\frac{p_1}{\alpha_5} = \frac{p_3
    \gamma_6}{\gamma_5}$, further $\frac{p_1}{\alpha_5} \gamma_5 = p_3
  \gamma_6$, so $\gamma_6 = \frac{p_1}{\alpha_5} \gamma_6'$ for some
  $\gamma_6'$. Next we get $q_2 = p_3\frac{p_1}{\alpha_5}
  \gamma_6'$. Using the equation for $\mu_2$ we conclude that
  $\gamma_2 p_2 = \frac{p_1}{\alpha_5} \gamma_6' q_3$. Hence $\gamma_2
  =  \frac{p_1}{\alpha_5} \gamma_2'$ for some $\gamma_2'$. Finally,
  $1 \neq \frac{p_1}{\alpha_5} \mid \gcd(\gamma_2,\gamma_6)$ which
  contradicts the Regulator Criterion for rank--$3$ groups. 

{\bf Case I.B.} The relevant invariants are 
  \[\textstyle
  \mu_1 = \frac{p_1}{\alpha_1} p_2 = \frac{q_1}{\beta_1} q_3, \quad
  \mu_6 = p_3 = q_3.
  \] 

The two equations together imply that $\frac{p_1}{\alpha_1} p_2 =
\frac{q_1}{\beta_1} p_3$, so $p_3 \mid p_1 p_2$, a contradiction as
the $p_i$ are pairwise relatively prime.

{\bf Case II.A.} The relevant invariants are 
  \[\textstyle
  \mu_1 = \frac{p_1}{\alpha_1} p_2 = \frac{q_1}{\beta_1} q_3, \quad 
  \mu_5 = p_2 = \frac{q_2}{\gamma_5}, \quad 
  \mu_6 = p_3 = \frac{q_2}{\gamma_6}.
  \] 

So $q_2 = \gamma_5 p_2$, and $q_2 = \gamma_6 p_3$.  It follows that
$q_2 = p_2 p_3 q_2'$ for some $q_2'$. As $p_1 p_2 p_3 = q_1 q_2 q_3$
it follows that $p_1 = q_1 q_2' q_3$. We now get from the equation for
$\mu_1$ that $\beta_1 (q_1 q_2' q_3) p_2 = \alpha_1 q_1 q_3$, hence
$\beta_1 q_2' p_2 = \alpha_1$. So $p_2$ divides $\alpha_1$ that is a
factor of $p_1$, a contradiction.

{\bf Case II.B.} The relevant invariants are 
  \[\textstyle
  \mu_1 = \frac{p_1}{\alpha_1} p_2 = \frac{q_1}{\beta_1} q_3, \quad
  \mu_5 = p_2 = \frac{q_2}{\gamma_5}, \quad \mu_6 = p_3 = q_3.
  \] 

So $q_2 = \gamma_5 p_2$. From $p_1 p_2 p_3 = q_1 q_2 q_3$, and the
equations for $\mu_5$ and $\mu_6$ we get that $p_1 p_2 p_3 = q_1
\gamma_5 p_2 p_3$, hence $p_1 = q_1 \gamma_5$. Substituting in the
equation for $\mu_1$ we get $\beta_1 (q_1 \gamma_5) p_2 = \alpha_1 q_1
p_3$ and further $\beta_1 \gamma_5 p_2 = \alpha_1 p_3$. Hence $p_2$
divides $\alpha_1$ that is a factor of $p_1$, again a contradiction.
\end{proof} 

\begin{ex}\label{S97fail}\rm The hooked family $\calS(9,7)$ cannot be
  realized in $\calB$.
\end{ex} 

\begin{proof} For easy reference 
  \[\textstyle
  \calS(9,7) = 
  \]
  \[
   \{(7,1,1),\, (6,2,1),\, (5,3,1),\,
  (5,2,2),\, (4,4,1),\,(4,3,2),\, (3,3,3),\, (7,2),\, (6,3) ,\,
  (5,4)\}.
  \]
  Suppose that $X \in \calB$ realizes $\calS(9,7)$. Then $X = [\tau_1
  \, \tau_2 \, \tau_3 \, \tau_4 \, \tau_5 \, \tau_6 \, \tau_7] \oplus
  [\tau_1] \oplus [\tau_2]$. The critical types $\tau_3, \ldots,
  \tau_7$ appear only once in decompositions.

  By assumption $X$ realizes $(5,2,2)$. Either one of $\tau_1, \tau_2$
  appears in the rank-$5$ summand, or both do.  We have two cases:
\begin{enumerate} 
\item {\bf Case I.} $X = [\tau_1 \, \tau_2 \, \tau_3 \, \tau_4 \,
  \tau_5] \oplus [\tau_1 \, \tau_6] \oplus [\tau_2 \, \tau_7]$ and
\item {\bf Case II.} $X = [\tau_1 \, \tau_3 \, \tau_4 \, \tau_5 \,
  \tau_6] \oplus [\tau_1 \, \tau_2] \oplus [\tau_2 \, \tau_7]$
\end{enumerate} 

{\bf Case I.} Explicitly we have
\begin{eqnarray*}
  X &=& X_1 \oplus X_2 \oplus X_3 \quad \text{where} \\ 
  X_1 &=& [\tau_1 v_1 \oplus \tau_2 v_2 \oplus \tau_3 v_3 \oplus
  \tau_4 v_4 \oplus \tau_5 v_5] + \frac{\Z}{p_1}(\alpha_1 v_1 +
  \alpha_2 v_2 + \alpha_3 v_3 + \alpha_4 v_4 + \alpha_5 v_5), \\  
  X_2 &=& [\tau_1 w_1 \oplus \tau_6 v_6] + \frac{\Z}{p_2}(w_1 + v_6), \\
  X_3 &=& [\tau_2 w_2 \oplus \tau_6 v_7] + \frac{\Z}{p_3}(w_2 + v_7),
\end{eqnarray*}
  and $\tau_1, \tau_6$ are $p_2$--free, and $\tau_2, \tau_7$ are $p_3$--free.
  We can now compute the invariants $\mu_i = \mu_{\tau_i}(X)$. 
  \[\textstyle
  \mu_1 = \frac{p_1}{\alpha_1} p_2, \quad 
  \mu_2 = \frac{p_1}{\alpha_2} p_3, \quad 
  \mu_3 = \frac{p_1}{\alpha_3}, \quad
  \mu_4 = \frac{p_1}{\alpha_4}, \quad 
  \mu_5 = \frac{p_1}{\alpha_5}, \quad 
  \mu_6 = p_2, \quad
  \mu_7 = p_3.
  \]

{\bf Case II.} Explicitly we have
\begin{eqnarray*}
  X &=& X_1 \oplus X_2 \oplus X_3 \quad \text{where} \\ 
  X_1 &=& [\tau_1 v_1 \oplus \tau_3 v_3 \oplus \tau_4 v_4 \oplus
  \tau_5 v_5 \oplus \tau_6 v_6] + \frac{\Z}{p_1}(\alpha_1 v_1 +
  \alpha_3 v_3 + \alpha_4 v_4 + \alpha_5 v_5 + \alpha_6 v_6), \\  
  X_2 &=& [\tau_1 w_1 \oplus \tau_2 v_2] + \frac{\Z}{p_2}(w_1 + v_2), \\
  X_3 &=& [\tau_2 w_2 \oplus \tau_7 v_7] + \frac{\Z}{p_3}(w_2 + v_7),
\end{eqnarray*}
  and $\tau_1, \tau_2$ are $p_2$--free, $\tau_2, \tau_7$ are $p_3$--free.
  We can now compute the invariants $\mu_i = \mu_{\tau_i}(X)$. 
  \[\textstyle
  \mu_1 = \frac{p_1}{\alpha_1} p_2, \quad 
  \mu_2 = p_2 p_3, \quad 
  \mu_3 = \frac{p_1}{\alpha_3}, \quad
  \mu_4 = \frac{p_1}{\alpha_4}, \quad 
  \mu_5 = \frac{p_1}{\alpha_5}, \quad
  \mu_6 = \frac{p_1}{\alpha_6}, \quad
  \mu_7 = p_3.
  \]

  Our group $X$ also realizes $(4,4,1)$. There is only one case as
  fully invariant summands are excluded.

\begin{enumerate} 
\item {\bf Case 441. } $X = [\tau_1 \, \tau_2 \, \tau_3 \, \tau_4]
  \oplus [\tau_2 \, \tau_5 \, \tau_6 \, \tau_7] \oplus [\tau_1]$
\end{enumerate} 

  Explicitly, 
\begin{eqnarray*}
  X &=& Y_1 \oplus Y_2 \oplus \tau_1 w_1 \quad \text{where} \\
  Y_1 &=& [\tau_1 v_1 \oplus \tau_2 v_2 \oplus \tau_3 v_3 \oplus
  \tau_4 v_4] + \frac{\Z}{q_1}(\beta_1 v_1 + \beta_2 v_2 + \beta_3
  v_3 + \beta_4 v_4), \\ 
  Y_2 &=& [\tau_2 w_2  \oplus \tau_5 v_5 \oplus \tau_6 v_6 \oplus \tau_7 v_7] +
  \frac{\Z}{q_2}(\gamma_2 w_2 + \gamma_5 v_5 + \gamma_6 v_6 + \gamma_7 v_7).
\end{eqnarray*}
  The invariants are
  \[\textstyle
  \mu_1 = \frac{q_1}{\beta_1}, \quad
  \mu_2 = \frac{q_1}{\beta_2} \frac{q_2}{\gamma_2}, \quad
  \mu_3 = \frac{q_1}{\beta_3}, \quad 
  \mu_4 = \frac{q_1}{\beta_4}, \quad
  \mu_5 = \frac{q_2}{\gamma_5}, \quad 
  \mu_6 = \frac{q_2}{\gamma_6}, \quad
  \mu_7 = \frac{q_2}{\gamma_7}.
  \]
\bigskip

{\bf Case I.441.} The invariants are

\begin{enumerate}
\item $\mu_1 = \frac{p_1}{\alpha_1} p_2 = \frac{q_1}{\beta_1}$,
\item $\mu_2 = \frac{p_1}{\alpha_2}  p_3 = \frac{q_1}{\beta_2}
  \frac{q_2}{\gamma_2}$, 
\item $\mu_3 = \frac{p_1}{\alpha_3} = \frac{q_1}{\beta_3}$,
\item $\mu_4 = \frac{p_1}{\alpha_4} = \frac{q_1}{\beta_4}$,
\item $\mu_5 = \frac{p_1}{\alpha_5} = \frac{q_2}{\gamma_5}$,
\item $\mu_6 = p_2 = \frac{q_2}{\gamma_6}$,
\item $\mu_7 = p_3 = \frac{q_2}{\gamma_7}$.
\end{enumerate} 
\medskip

From (6) and (7) it follows that $q_2 = p_2 p_3 p_1'$ for some $p_1'$, 
and $p_1 = q_1 p_1'$. It then follows from (1) that
$\frac{p_1}{\alpha_1} p_2 = \frac{p_1'}{\beta_1}$, so $p_2 \mid  p_1'
\mid p_1$, a contradiction.

This excludes the case I.

$X$ realizes $(3,3,3)$. 
\begin{enumerate} 
\item {\bf Case 333. } $X = [\tau_1 \, \tau_2 \, \tau_3] \oplus
  [\tau_1 \, \tau_4 \, \tau_5] \oplus [\tau_2 \, \tau_6 \, \tau_7]$.
\end{enumerate} 
Explicitly, $X = Y_1 \oplus Y_2 \oplus Y_3$ where 
  \[\textstyle
  Y_1 = [\tau_1 v_1 \oplus \tau_2 v_2 \oplus \tau_3 v_3] +
  \frac{\Z}{q_1} (\beta_1 v_1 + \beta_2 v_2 + \beta_3 v_3),
  \] 
  \[\textstyle
  Y_2 = [\tau_1 w_1 \oplus \tau_4 v_4 \oplus \tau_5 v_5] +
  \frac{\Z}{q_2} (\gamma_1 w_1 + \gamma_4 v_4 + \gamma_5 v_5),
  \] 
  \[\textstyle
  Y_3 =  [\tau_2 w_2 \oplus \tau_6 v_6 \oplus \tau_7 v_7] +
  \frac{\Z}{q_3} (\delta_2 w_2 + \delta_6 v_6 + \delta_7 v_7).
  \] 
{\bf Case II.333. } The invariants are 
\begin{enumerate}
\item $\mu_1 = \frac{q_1}{\beta_1}\frac{q_2}{\gamma_1} =
  \frac{p_1}{\alpha_1} p_2$,
\item $\mu_2 = \frac{q_1}{\beta_2}\frac{q_3}{\delta_2} = p_2 p_3$, 
\item $\mu_3 = \frac{q_1}{\beta_3} = \frac{p_1}{\alpha_3}$,
\item $\mu_4 = \frac{q_2}{\gamma_4} = \frac{p_1}{\alpha_4}$,
\item $\mu_5 = \frac{q_2}{\gamma_5} = \frac{p_1}{\alpha_5}$,
\item $\mu_6 = \frac{q_3}{\delta_6} = \frac{p_1}{\alpha_6}$,
\item $\mu_7 = \frac{q_3}{\delta_7} = p_3$.
\end{enumerate} 

Recall that $q_1 q_2 q_3 = p_1 p_2 p_3$. We have the following
consequences. 
\begin{enumerate}
\item[{[$\mu_7$]}] $q_3 = p_3 \delta_7$. Hence $q_1 q_2 \delta_7 = p_1
  p_2$. 
\item[{[$\mu_1$]}] $q_1 q_2 \alpha_1 = p_1 p_2  \beta_1 \gamma_1 = q_1
  q_2 \delta_7 \beta_1 \gamma_1$, 
  hence $\alpha_1 = \delta_7 \beta_1 \gamma_1$ and further, $\delta_7,
  \beta_1, \gamma_1$ all divide $p_1$.
\item[{[$\mu_2$]}] $q_1 p_3 \delta_7 = q_1 p_3 \delta_7 = p_2 p_3 \beta_2
  \delta_2$, hence $q_1 \delta_7 = p_2 \beta_2 \delta_2$. As $\delta_7
  \mid p_1$, it follows that $\beta_2 \delta_2 = \delta_7 \omega$ for
  some $\omega$ and hence $q_1 = p_2 \omega$.
\item[{[*]}] $q_1 q_2 q_3 = p_1 p_2 p_3$ so $p_2 \omega q_2 p_3 \delta_7 =
  p_1 p_2 p_3$, and $\omega q_2 \delta_7 = p_1$.
\item[{[$\mu_3$]}] $\alpha_3 q_1 = p_1 \beta_3$, so
  $\alpha_3 p_2 \omega = p_1 \beta_3$, $\beta_3 = p_2 \beta_3'$, and 
  $\alpha_3 \omega = p_1 \beta_3'$. Also $\alpha_3 q_1 = \omega q_2
  \delta_7 \beta_3$, hence $\alpha_3 = q_2 \alpha_3'$, and $q_2
  \alpha_3' q_1 = \omega q_2 \delta_7 \beta_3$, so  $\alpha_3' q_1 =
  \omega \delta_7 \beta_3$, $\alpha_3' p_2 \omega = \omega \delta_7
  p_2 \beta_3'$, thus $\alpha_3' = \delta_7 \beta_3'$.
\item[{[$\mu_4$]}] $\alpha_4 q_2 = p_1 \gamma_4 = \omega q_2 \delta_7
  \gamma_4$, so $\alpha_4 = \omega \delta_7 \gamma_4$.
\item[{[$\mu_5$]}] $q_2 \alpha_5 = p_1 \gamma_5 = \omega q_2
  \delta_7\gamma_5$, hence $\alpha_5 = \omega \delta_7 \gamma_5$. 
\item[{[$\mu_6$]}] $\alpha_6 q_3 = p_1 \delta_6 = \omega q_2 \delta_7
  \delta_6$, hence $\alpha_6 = q_2 \alpha_6'$ and $\alpha_6' q_3 =
  \omega \delta_7 \delta_6$ and $\alpha_6' p_3 \delta_7 = \omega
  \delta_7 \delta_6$, so $\alpha_6' p_3 = \omega \delta_6$.
\end{enumerate} 

We summarize the results. 
\[
q_3 = p_3 \delta_7, \quad
q_1 = p_2 \omega, \quad p_1 = \omega q_2 \delta_7.
\]
We seek the most general solution to the equations with given values
$p_1, p_2, p_3, q_2, \delta_7, \omega$ where $p_1 = q_2 \delta_7
\omega$ with pairwise relatively prime factors. Then
  \[
  \alpha_1 = \delta_7 \beta_1 \gamma_1 \quad 
  \alpha_3 = q_2 \delta_7 \beta_3' \quad 
  \alpha_4 = \delta_7 \omega \gamma_4 \quad 
  \alpha_5 = \delta_7 \omega \gamma_5 \\ 
  \alpha_6 = q_2 \alpha_6' \\ 
  \] 
We see that four of the five $\alpha$ contain the factor $\delta_7$,
and the Regulator Criterion requires that $\delta_7 = 1$. Using this
fact we have 
  \[\textstyle
  q_3 = p_3 \quad q_1 = p_2 \omega \quad p_1 = \omega q_2 \quad 
  q_1 q_2 = p_1 p_2.
  \] 
  So $\omega \mid q_1$, hence $\gcd(\omega, q_3) = 1$ and
  $\gcd(\omega, p_3) = 1$ as $q_3 = p_3$. It then follows from
  $\alpha_6' p_3 = \omega \delta_6$ ([$\mu_6$]) that $q_3 = p_3 \mid
  \delta_6$ which is a contradiction because this makes $Y_3$
  decomposable.
\end{proof}

\begin{ex} \rm There is a group in $\calB$ realizing $(5,2,2)$,
  $(5,3,1)$, $(6,2,1)$, $(7,1,1)$ and $(4,3,2)$.
  
 Let 
\begin{eqnarray*}
  X &=& [(\tau_1 v_1 \oplus \tau_3 v_3 \oplus \tau_4 v_4 \oplus
  \tau_5 v_5 \oplus \tau_6 v_6) + \frac{\Z}{3 \cdot 5}(v_1 +
  3 v_3 + 3 v_4 + 5 v_5 + 5 v_6)] \\  
   &\oplus& [(\tau_1 w_1 \oplus \tau_2 v_2) + \frac{\Z}{7}(w_1 + v_2)] \\
   &\oplus& [(\tau_2 w_2 \oplus \tau_7 v_7) + \frac{\Z}{11}(w_2 +
   v_7)] \\
  Y &=& [(\tau_1 v_1 \oplus \tau_2 v_2 \oplus \tau_3 v_3 \oplus
  \tau_4 v_4) + \frac{\Z}{7 \cdot 5} (v_1 + 5 v_2 + 7 v_3
  + 7 v_4)] \\  
  &\oplus& [(\tau_1 w_1 \oplus \tau_5 v_5 \oplus \tau_6 v_6)
  + \frac{\Z}{3} (w_1 + v_5 + v_6)] \\
  &\oplus& [(\tau_2 w_2 \oplus \tau_7 v_7)
  + \frac{\Z}{11} (w_2 + v_7)].
\end{eqnarray*}
It is clear the $\R(X) \cong \R(Y)$ and it is easily checked that
$\forall\, i : \mu_i(X) = \mu_i(Y)$. Hence $X \nriso Y$ and $X$
realizes both $(5,2,2)$ and $(4,3,2)$. It follows from \lemref{sums
    of rigid crq groups} that $X$ also realizes $(5,3,1)$, $(6,2,1)$,
and $(7,1,1)$.
\end{ex}

\end{document}